\documentclass[11pt]{amsart}

\usepackage{amsmath,amssymb,latexsym,soul,cite,amsthm,color,enumitem,graphicx,tikz, mathtools,microtype}
\usepackage{amsfonts,amsthm,amssymb,marvosym,esint,xfrac,bbm,amssymb}

\usepackage[normalem]{ulem}

\usepackage[colorlinks=true, urlcolor=blue, citecolor=red, linkcolor=blue]{hyperref}
\usepackage[english]{babel}
\usepackage[left=2.5cm,right=2.5cm,top=2.7cm,bottom=2.7cm]{geometry}
\usepackage{cancel}

\numberwithin{equation}{section}

\newtheorem{theorem}{Theorem}[section]
\theoremstyle{plain}
\newtheorem{lemma}[theorem]{Lemma}
\theoremstyle{plain}
\newtheorem{proposition}[theorem]{Proposition}
\theoremstyle{plain}

\newtheorem{definition}[theorem]{Definition}
\theoremstyle{definition}
\newtheorem{remark}[theorem]{Remark}

\def\Xint#1{\mathchoice
    {\XXint\displaystyle\textstyle{#1}}%
    {\XXint\textstyle\scriptstyle{#1}}%
    {\XXint\scriptstyle\scriptscriptstyle{#1}}%
    {\XXint\scriptscriptstyle\scriptscriptstyle{#1}}%
    \!\int}
\def\XXint#1#2#3{\setbox0=\hbox{$#1{#2#3}{\int}$}
    \vcenter{\hbox{$#2#3$}}\kern-0.5\wd0}
\def\bint{\Xint-}
\def\bint{\Xint-}
\def\dashint{\Xint{\raise4pt\hbox to7pt{\hrulefill}}}
\def\dashiint{\bint\kern-0.15cm\bint}

\title{Harnack-type  inequalities for nonlocal doubly nonlinear parabolic equations}
\author{M. Savchenko$^{1}$,  I. Skrypnik$^{1}$, Y. Yevgenieva$^{2,1}$}
\thanks{$^{1}$Institute of Applied Mathematics and Mechanics, National Academy of Sciences of Ukraine, 84116 Sloviansk, Ukraine}
\thanks{$^{2}$Max Planck Institute for Dynamics of Complex Technical Systems, 39106 Magdeburg, Germany}
\thanks{{\tt\small\ shan{\_}maria@ukr.net, ihor.skrypnik@gmail.com, yevgenieva@mpi-magdeburg.mpg.de}}

\begin{document}
\begin{abstract} 
We establish intrinsic Harnack inequalities for nonlocal doubly nonlinear parabolic equations in the slow-diffusion range:
\begin{equation*}
(u^q)_t+ \operatorname{P.V.} \int\limits_{\mathbb{R}^N}\frac{|u(x, t)-u(y, t)|^{p-2}(u(x, t)-u(y, t))}{|x-y|^{N+sp}}\,dy=0,\quad s\in(0,1),\quad  0<q < p-1.
\end{equation*}



\end{abstract}

\maketitle

	\begin{center}
		\begin{minipage}{9cm}
			\small
			\tableofcontents
		\end{minipage}
	\end{center}

\addtocontents{toc}{\protect\setcounter{tocdepth}{1}}

\section{Introduction and Main Results}
In this paper, we consider nonnegative solutions to the equation
\begin{equation}\label{eq1.1}
(u^q)_t+\mathcal{L}u= 0,\quad (x, t)\in \Omega_T,\quad 0< q <p-1,
\end{equation}
where $\Omega_T:=\Omega \times(0, T)$, $\Omega$ is a bounded open domain in $\mathbb{R}^N$, $u\geqslant 0$ in $\Omega_T$. The nonlocal operator $\mathcal{L}$ is defined by
\begin{equation}\label{eq1.2}
\mathcal{L} u= \operatorname{P.V.} \int\limits_{\mathbb{R}^N} K(x, y, t) |u(x, t)- u(y, t)|^{p-2}(u(x, t)- u(y, t))\,dy,
\end{equation}
for some $p>1$. Here $\operatorname{P.V.}$ denotes the principal value of the integral, whereas the kernel
$K : \mathbb{R}^n \times \mathbb{R}^N \times (0, T) \rightarrow \mathbb{R}_+$ is a measurable function satisfying the following condition uniformly in $t$:
uniformly in $t$:
\begin{equation}\label{eq1.3}
\frac{K_1}{|x-y|^{N+sp}} \leqslant K(x, y, t)=K(y, x, t)\leqslant \frac{K_2}{|x-y|^{N+sp}},\quad \text{a. e.}\quad x, y \in \mathbb{R}^N, 
\end{equation}
with some $0< K_1 < K_2$ and $s\in (0, 1)$, $sp<N$.  We refer to the parameters $K_1$, $K_2$, $p$, $q$, $s$, $N$ as the {\it data}. Throughout the paper, $\gamma$ denotes a generic positive constant depending only on the data, which may change from line to line.

Postponing the formal definitions of weak solution, we will proceed to present the main results. 

\begin{theorem}\label{th1.1}
Let $u$ be a  continuous local weak solution to \eqref{eq1.1}--\eqref{eq1.3} in $\Omega_T$,  $u\geqslant 0$ in $\Omega_T$ and let $0< q< p-1$.  Fix $(x_0, t_0)\in \Omega_T$ such that $u(x_0, t_0) > 0$. There exist positive constants $C_*$, $C_1$, $C_2$, $C_3$ depending only on the data, such that, if
$$Q_{16\rho,4 C_*\theta}(x_0, t_0)
:=B_{16\rho}(x_0)\times
(t_0-4\,C_*\theta, t_0+4 C_*\theta)
\subset\Omega_T,
\qquad
\theta:= u(x_0, t_0)^{q-p+1}\rho^{sp},
$$
and if
\begin{equation}\label{eqHarnackTail}
\left(\,\,
\int\limits_{t_0- 4 C_*\theta}^{t_0+4 C_*\theta}
\,\,\int\limits_{\mathbb{R}^N\setminus B_\rho(x_0)}
\frac{[u_-(x,t)]^{p-1}}
{|x- x_0|^{N+sp}}
dx dt
\right)^{\frac1q}
\leqslant
\frac{1}{C_1}\,u(x_0, t_0),
\end{equation}
then the following estimates hold:
\begin{equation}\label{eqHarnack1}
\frac{1}{C_2}\sup\limits_{B_\rho(x_0)}u(\cdot, t_0-C_* \theta)\leqslant u(x_0, t_0)\leqslant C_2 \inf\limits_{B_\rho(x_0)}u(\cdot, t_0+C_* \theta),
\end{equation}
and
\begin{equation}\label{eqHarnack2}
\begin{aligned}
 \frac{1}{C_3}\left(\int\limits_{t_0- C_*\theta}^{t_0-\frac{1}{2}C_*\theta}
\,\,\int\limits_{\mathbb{R}^N}\frac{[u_+(x,t)]^{p-1}}{(\rho+|x-x_0|)^{N+sp}} dx dt\right)^{\frac1q}&\leqslant u(x_0, t_0)
\\&\leqslant C_3 \left(\,\,\int\limits_{t_0- \frac{1}{4}C_*\theta}^{t_0}
\,\,\int\limits_{\mathbb{R}^N}\frac{[u_+(x,t)]^{p-1}}{(\rho+|x-x_0|)^{N+sp}} dx dt\right)^{\frac1q}.
\end{aligned}
\end{equation}
\end{theorem}

Theorem \ref{th1.1} is a consequence of the following two propositions. The first one is an analogue of the weak Harnack inequality. 

\begin{proposition}\label{pr1.1}
Let $u$ be a  local weak super-solution of \eqref{eq1.1}--\eqref{eq1.3} in $\Omega_T$,  $u\geqslant 0$ in $\Omega_T$.  Assume that $0< q< p-1$ and fix number $\mathcal{I}>0$. There exist constants $\bar{C}_*$, $\bar{C}_1$, $\bar{C}_2>1$, depending only on the data, such that for every cylinder
$$
Q_{16\rho,4\bar{C}_*\theta}(\bar{x},\bar{t})
:=B_{16\rho}(\bar{x})\times
(\bar{t}-4\bar{C}_*\theta,\bar{t}+4\bar{C}_*\theta)
\subset\Omega_T,
\qquad
\theta:=\mathcal{I}^{q-p+1}\rho^{sp},
$$
for which
\begin{equation}\label{eqWeakProp1}
\mathcal{I}\leqslant
\left(\,\,
\int\limits_{\bar{t}-\theta}^{\bar{t}}
\,\,\int\limits_{\mathbb{R}^N}
\frac{[u_+(x,t)]^{p-1}}
{(\rho+|x-\bar{x}|)^{N+sp}}
dx dt
\right)^{\frac1q},
\end{equation}
and
\begin{equation}\label{eqWeakProp2}
\left(\,\,
\int\limits_{\bar{t}-4\bar{C}_*\theta}^{\bar{t}+4\bar{C}_*\theta}
\,\,\int\limits_{\mathbb{R}^N\setminus B_\rho(\bar{x})}
\frac{[u_-(x,t)]^{p-1}}
{|x-\bar{x}|^{N+sp}}
dx dt
\right)^{\frac1q}
\leqslant
\frac{\mathcal{I}}{\bar{C}_1}.
\end{equation}
Then
\begin{equation}\label{eqWeakProp3}
\mathcal{I}\leqslant
\bar{C}_2\inf\limits_{B\rho(\bar{x})}
u(\cdot,\bar{t}+\bar{C}_*\theta).
\end{equation}
\end{proposition}

\begin{remark}\label{remWeak}
Note that since 
$$\int\limits^{\bar{t}}_{\bar{t}-\theta}\,\,\int\limits_{B_\rho(\bar{x})}\frac{\big[u(x, t)\big]^{p-1}}{\big(\rho+|x-\bar{x}|\big)^{N+sp}}\,dx dt\asymp \frac{1}{\rho^{N+sp}} \int\limits^{\bar{t}}_{\bar{t}-\theta}\,\,\int\limits_{B_\rho(\bar{x})}\big[u(x, t)\big]^{p-1}\,dx dt,$$
and
$$\int\limits^{\bar{t}}_{\bar{t}-\theta}\,\,\int\limits_{\mathbb{R}^N \setminus B_\rho(\bar{x})}\frac{\big[u_+(x, t)\big]^{p-1}}{\big(\rho+|x-\bar{x}|\big)^{N+sp}}\,dx dt\asymp  \int\limits^{\bar{t}}_{\bar{t}-\theta}\,\,\int\limits_{\mathbb{R}^N\setminus B_\rho(\bar{x})}\frac{\big[u_+(x, t)\big]^{p-1}}{|x-\bar{x}|^{N+sp}}\,dx dt,$$
condition \eqref{eqWeakProp1} is equivalent (up to the constant) to
\begin{equation}\label{eqWeakProp4}
\mathcal{I}\leqslant \Big(\fint\limits^{\bar{t}}_{\bar{t}-\theta}\,\fint\limits_{B_\rho(\bar{x})}\big[u(x, t)\big]^{p-1}\,dx dt\Big)^{\frac{1}{p-1}}+ \operatorname{Tail}(u_+, Q^-_{\rho, \theta}(\bar{x}, \bar{t})),\quad \theta=\mathcal{I}^{q-p+1}\rho^{sp}.
\end{equation}
\end{remark}

\begin{proposition}\label{pr1.2}
Let $u$ be a  local weak sub-solution to \eqref{eq1.1}-\eqref{eq1.3} in $\Omega_T$, $u\geqslant 0$ in $\Omega_T$, and  $0<q < p-1$. In the case $0< q< 1$ assume additionally that $u$ is locally bounded in $\Omega_T$. Then there exists number $\bar{C}_3 > 0$ depending only on the data, such that, for every cylinder
$Q^-_{\rho, \theta}(\bar{x}, \bar{t})\subset \Omega_T$, there holds
\begin{equation}\label{eqUp1}
\sup\limits_{Q^-_{\frac{1}{2}\rho, \frac{1}{2}\theta}(\bar{x}, \bar{t})} u\leqslant  \bar{C}_3 \bigg[\bigg( \int\limits^{\bar{t}}_{\bar{t}-\theta}\,\int\limits_{\mathbb{R}^N}\frac{[u_+(x, t)]^{p-1}}{(\rho+|x-\bar{x}|)^{N+sp}}dx d t\bigg)^{\frac{1}{q}}+
 \Big(\frac{\rho^{sp}}{\theta}\Big)^{\frac{1}{p-1-q}}\bigg].
\end{equation}
\end{proposition}

\subsection{Novelty and significance}

In this subsection, we provide an overview of the development of the regularity theory for local and nonlocal doubly nonlinear equations, place our main result in this broader context, and discuss the main ideas and techniques used in its proof.

The local counterpart of \eqref{eq1.1} is the following doubly nonlinear parabolic equation:
\begin{equation}\label{eq:loc_dn_eq}
(|u|^{q-1}u)_t-\operatorname{div}(|\nabla u|^{p-2}\nabla u)=0,\qquad p>1,\ q>0.
\end{equation}
The regularity theory for \eqref{eq:loc_dn_eq} is by now well developed, although some regimes have only recently been understood. For the Trudinger case $q=p-1$, Harnack theory goes back to~\cite{Tru1968} and was developed in~\cite{GiaVes2006,KinKuu2007}. H\"older continuity was established in~\cite{KuuSilUrb2012,KuuLalSilUrb2012} and extended to signed solutions in~\cite{BogDuzLia2021}.

For general doubly nonlinear equations, early H\"older estimates in the doubly degenerate regime ($p>2$ and $q<p-1$) were obtained in~\cite{Iva1991,PorVes1993}, while the doubly singular regime ($p<2$ and $q>p-1$) was studied in~\cite{VesVes2022,Hen2023}. Results for signed solutions were obtained in~\cite{LiaSch2022,BogDuzLiaSch2023}. In the mixed regimes (degenerate--singular or singular--degenerate), the competing effects of degeneracy and singularity make the propagation of positivity and the reduction of oscillation more delicate. The recent works~\cite{CiaHenSavSkrYev2026arxiv} combines a new integral $L^1$--$L^1$ Harnack estimate with intrinsic scaling and positivity arguments to establish H\"older continuity for nonnegative solutions up to the relevant Barenblatt thresholds. Signed solutions are studied in~\cite{Skr2026arxiv}.

Concerning Harnack inequalities, B\"ogelein, Heran, Sch\"atzler, and Singer~\cite{BogHerSchSin2021} established an intrinsic Harnack inequality throughout the slow-diffusion range, $p>1$ and $0<q<p-1$. B\"ogelein, Duzaar, Gianazza, Liao, and Scheven~\cite{BogDuzGiaLiaSch2026} obtained forward, backward, and elliptic Harnack estimates in the fast-diffusion range $p>1$, $p-1<q<\frac{N(p-1)}{(N-p)_+}$.

For the nonlocal equation~\eqref{eq1.1}, the available results depend strongly on the particular form of the nonlinearity. When $q=1$ and $p>2$, Prasad~\cite{Pra2026arxiv} established an elliptic-type Harnack inequality for the fractional parabolic $p$-Laplacian by means of a comparison principle. Liao~\cite{Lia2026arxiv} obtained intrinsic Harnack estimates and, in addition, established an expansion of positivity result.
In the singular range $1<p<2$, Cassanello, Ciani, and Iannizzotto~\cite{CasCiaIan2026} obtained integral Harnack-type estimates.

In the fast-diffusion setting ($p=2$ and $q>1$), Cassanello, De Filippis, and Stanko~\cite{CasDeFSta2026arxiv} proved an integral Harnack inequality for $1<q<\frac{N}{N-2s}$, with $N\ge2$, and an expansion of positivity for $q>1$, under suitable tail assumptions. The integral estimate is a step towards a full pointwise Harnack theory.

For the nonlocal Trudinger equation, Prasad~\cite{Pra2024} proved a weak Harnack inequality for globally nonnegative weak supersolutions. Shang and Zhang~\cite{ShaZha2026} obtained weak Harnack estimates with explicit tail contributions, allowing nonnegativity to be imposed locally. Adimurthi~\cite{Adi2025arxiv} establishes H\"older continuity for bounded, possibly sign-changing weak solutions. Ciani and Nakamura~\cite{CiaNak2026arxiv} established a full Harnack inequality for $p>1$, with quantitative control of the nonlocal contributions.

Beyond the balanced case, that is, when $q\neq p-1$, the regularity theory remains less complete. Li~\cite{Li2026} establishes H\"older continuity for locally bounded, possibly sign-changing weak solutions in the doubly degenerate range $p>2$, $0<q<p-1$. Adimurthi and Modasiya~\cite{AdiMod2025arxiv} obtain H\"older estimates for locally bounded signed solutions for a restricted range of mixed regmes. Using the approach of DiBenedetto, Gianazza, and Vespri~\cite{DiBGiaVes2012}, Misawa and Yamaura~\cite{MisYam2025} established expansion of positivity under appropriate tail conditions. A unified theory of weak and full Harnack inequalities for genuinely doubly nonlinear nonlocal equations remains to be developed.


The main result of this paper, Theorem~\ref{th1.1}, establishes an intrinsic Harnack inequality for locally bounded, nonnegative local weak solutions to \eqref{eq1.1}-\eqref{eq1.3} in the slow-diffusion range $0<q<p-1$. 

Propositions \ref{pr1.1} and \ref{pr1.2} play distinct but complementary roles in the proof of Theorem~\ref{th1.1}. The weak Harnack estimate controls the infimum at later times by a weighted integral over an earlier time interval, while the local upper estimate controls the supremum in a backward cylinder by a weighted integral of the solution. Applying these estimates on appropriately chosen intrinsic cylinders yields the two-sided integral bounds in \eqref{eqHarnack2}. Their combination also gives the pointwise comparison in \eqref{eqHarnack1}. Consequently, the proof of the main result reduces to establishing these two one-sided estimates, which are of independent interest for the regularity theory of doubly nonlinear nonlocal parabolic equations.

\subsection{Structure of the paper}

The paper is organized as follows. Section~\ref{sec:aux} introduces the notation and the definition of a weak solution to \eqref{eq1.1} and presents preliminary results, including local energy estimates and De Giorgi-type lemmas. Section~\ref{sec:ExpPos} establishes the expansion of positivity, a key ingredient in the proof of the main result. Sections~\ref{sec:prop1.1} and~\ref{sec:prop1.2} are devoted to the proofs of Propositions~\ref{pr1.1} and~\ref{pr1.2}, respectively. Finally, Section~\ref{sec:mainProof} combines these estimates to prove Theorem~\ref{th1.1}.

\section{Auxiliary Propositions}\label{sec:aux}

\subsection{Definitions and Notations}

For any $a$, $b > 0$ we  write $a\asymp b$ if there exists $\gamma >0$ such that
$$\gamma^{-1}\, b \leqslant a \leqslant \gamma\,b.$$

\subsubsection{\textbf{Function Spaces}}

Let $p>1$ and $s\in(0,1)$. We define the fractional Sobolev space $W^{s,p}(\mathbb{R}^N)$ by
$$W^{s, p}(\mathbb{R}^N):=\big\{ u\in L^p(\mathbb{R}^N) : \int\limits_{\mathbb{R}^N}\int\limits_{\mathbb{R}^N}\frac{|u(x)-u(y)|^p}{|x-y|^{N+sp}}\,dx dy < \infty\big\},$$
equipped with the norm 
$$||u||_{W^{s, p}(\mathbb{R}^N)}:= \Big(\int\limits_{\mathbb{R}^N}|u|^p dx \Big)^{\frac{1}{p}}+
\bigg(\int\limits_{\mathbb{R}^N}\int\limits_{\mathbb{R}^N}\frac{|u(x)-u(y)|^p}{|x-y|^{N+sp}}\,dx dy\bigg)^{\frac{1}{p}}.$$
For a domain $\Omega\subset\mathbb{R}^N$, the corresponding fractional
Sobolev space $W^{s, p}(\Omega)$ is defined analogously. We also denote
$$W^{s, p}_0(\Omega):=\big\{ u\in W^{s, p} (\mathbb{R}^N): u=0 \quad \text{a.e. in}\quad\mathbb{R}^N\setminus\Omega\big\}.$$
In what follows, we need the following parabolic Sobolev embedding theorem (see, e.g., \cite[Proposition A.3]{Lia2024}).

\begin{lemma}\label{lemEmbedding}
Let $p\geqslant1$  and $s\in(0,1)$,  $sp<N$. For any function $u\in L^p(t_1, t_2; W^{s, p}(B_r(x_0)))\cap L^\infty (t_1, t_2; L^{2}(B_r(x_0)))$, compactly supported 
in $B_{r(1-\sigma)}(x_0)$ with some $\sigma \in (0, 1)$ and for  a.e.  $t\in (t_1, t_2)$, there holds
\begin{equation*}
\begin{aligned}
&\int\limits^{t_2}_{t_1}\int\limits_{B_r(x_0)}|u|^{p\left(1+\frac{2s}{N}\right)}dx dt
\leqslant 
\gamma \biggr(\sup\limits_{t\in (t_1, t_2)}\int\limits_{B_r(x_0)}|u|^{2}\,dx\biggr)^{\frac{sp}{N}}
\\&\times
\biggr(\int\limits^{t_2}_{t_1}\int\limits_{B_r(x_0)}\int\limits_{B_r(x_0)}
\frac{|u(x, t)-u(y, t)|^p}{|x-y|^{N+sp}}dx dy dt+\frac{1}{\sigma^{N+sp} r^{sp}}\int\limits^{t_2}_{t_1}\int\limits_{B_r(x_0)}|u|^p dx dt\biggr),
\end{aligned}
\end{equation*}
where $\gamma>0$ depends only on $N$, $s$, and $p$. 
\end{lemma}

\subsubsection{\textbf{Notion of Weak Solution}}

\begin{definition}
A measurable function
$u:\mathbb{R}^N\times(0,T]\rightarrow\mathbb{R}$ satisfying
\begin{equation}\label{eqDefinition1}
u\in C_{loc}(0, T; L^{q+1}_{loc}(\Omega))\cap L^p_{loc}(0, T; W^{s, p}_{loc}(\Omega)),\quad u\geqslant 0 \quad\text{in}\quad \Omega_T,
\end{equation}
is called a local weak sub(super)-solution to \eqref{eq1.1}, with
the kernel $K$ satisfying \eqref{eq1.3}, if the following conditions
hold. For every compact subset $E\subset\Omega$ and every
subinterval $[t_1,t_2]\subset(0,T]$, we have
\begin{equation}\label{eqDefinition2}
\int\limits_{t_1}^{t_2}\int\limits_{\mathbb{R}^N}\frac{|u(x, t)|^{p-1}}{1+|x|^{N+sp}}\,dx dt <\infty,
\end{equation}
and
\begin{equation}\label{eqDefinition3}
\begin{aligned}
&\int\limits_{E}u^{q} \varphi dx\Big|^{t_2}_{t_1}-\int\limits^{t_2}_{t_1}\int\limits_{E} u^q \varphi_t\,dx dt
\\&+\int\limits^{t_2}_{t_1}\int\limits_{\mathbb{R}^N}\int\limits_{\mathbb{R}^N}K(x, y, t) |u(x, t)-u(y, t)|^{p-2}(u(x, t)-u(y, t))(\varphi(x, t)-\varphi(y, t))\,dx dy dt\leqslant ( \geqslant) 0,
\end{aligned}
\end{equation}
for every nonnegative test function
$$\varphi \in W^{1,q+1}_{loc}(0, T; L^{q+1}(E))\cap L^p_{loc}(0, T; W^{s, p}_0(E)).$$

A function that is both a local weak sub-solution and a local weak
super-solution is called a local weak solution.
\end{definition}

We also introduce the following tail term, which naturally arises in the nonlocal setting:
\begin{equation}\label{eqTail}
\operatorname{Tail}\big(u, Q_{\rho, \theta}(x_0, t_0)\big):= \biggr(\int\limits^{t_0}_{t_0-\theta}\,\,\,\int\limits_{\mathbb{R}^N\setminus B_\rho(x_0)}\frac{|u(x, t)|^{p-1}}{|x-x_0|^{N+sp}}\,dxdt\biggr)^{\frac{1}{q}},
\end{equation}
where 
$$Q_{\rho, \theta}(x_0, t_0):=B_\rho(x_0)\times (t_0-\theta, t_0).$$


\subsection{Algebraic Lemma}

We further recall the following algebraic estimates. A proof of inequality \eqref{eqalg1} can be found in \cite[Lemma 2.2]{AceFus1989} for $0<p<1$ and in \cite[inequality (2.4)]{GiaMod1986} for $p>1$.
\begin{lemma}\label{lemAlg}
There hold
 \begin{equation}\label{eqalg1}
 ||a|^{p-1}a -|b|^{p-1} b|\asymp\,\big(|a|+|b|\big)^{p-1}|a-b|,\quad p>0, \quad a, b\in \mathbb{R},
 \end{equation}
 \begin{equation}\label{eqalg2}
 a^p-b^p\leqslant p a^{p-1}(a-b)\leqslant \epsilon a^p+\frac{c}{\epsilon^{p-1}} (a-b)^p,\quad p>1,\quad \epsilon>0,\quad a>b>0.
 \end{equation}

\end{lemma}

\subsection{Energy Estimates}

\begin{lemma}\label{lemEnergy}
Let $u$ be a local weak sub(super)-solution to \eqref{eq1.1}-\eqref{eq1.3} in $\Omega_T$ and let $k(t)$ be absolutely continuous in $(0, T)$.
Then for every cylinder $Q_{8r, 8\eta}:=Q_{8r, 8\eta}(x_0, t_0)\subset \Omega_T$ and every piecewise smooth cutoff function $\zeta(x, t)$ compactly supported in $B_r:=B_r(x_0)$ for all
$t\in (t_0-\eta, t_0)$ there holds
\begin{equation}\label{eqenergy}
\begin{aligned}
\int\limits_{B_r\times\{\tau\}}\!\!\!g(u^{\pm}_{k(\tau)})\zeta^p\,dx\Big|^{t_0}_{t_0-\eta}&+\int\limits_{t_0-\eta}^{t_0}\int\limits_{B_r}\int\limits_{B_r}
\frac{|u^{\pm}_{k(\tau)}(x, \tau)-u^{\pm}_{k(\tau)}(y, \tau)|^p}{|x-y|^{N+sp}}\min[\zeta(x, \tau), \zeta(y, \tau)]^p\,dx dy d\tau\\
&+\int\limits^{t_0}_{t_0-\eta}\int\limits_{B_r}u^{\pm}_{k(\tau)}(x, \tau)\zeta^p(x, \tau)\,\bigg(\int\limits_{\mathbb{R}^N}\frac{\big[u^{\mp}_{k(\tau)}(y, \tau)\big]^{p-1}}{|x-y|^{N+sp}}dy\bigg)\,dxd\tau\\
&\leqslant \mp\iint\limits_{Q_{r, \eta}}k'(\tau)h(u^{\pm}_{k(\tau)})\zeta^p(x, \tau) dxd\tau+\gamma \iint\limits_{Q_{r, \eta}}g(u^{\pm}_{k(\tau)})\zeta^{p-1}(x, \tau)|\zeta_\tau| dxd\tau\\
&+\gamma\int\limits^{t_0}_{t_0-\eta}\int\limits_{B_r}\int\limits_{B_r}\max[u^{\pm}_{k(\tau)}(x, \tau), u^{\pm}_{k(\tau)}(y, \tau)]^p\frac{|\zeta(x, \tau)-\zeta(y, \tau)|^p}{|x-y|^{N+sp}}\,dxdyd\tau\\
&+\int\limits^{t_0}_{t_0-\eta}\int\limits_{B_r}u^{\pm}_{k(\tau)}(x, \tau)\zeta^p(x, \tau)\bigg(\int\limits_{\mathbb{R}^N\setminus B_r}\frac{\big[u^{\pm}_{k(\tau)}(y, \tau)\big]^{p-1}}{|x-y|^{N+sp}} dy\bigg) dxd\tau.
\end{aligned}
\end{equation}
Here
$$u^\pm_{k(t)}=(u-k(t))_\pm,\quad g(u^{\pm}_{k(\tau)}):=\int\limits^{u^{\pm}_{k(\tau)}}_0(z\pm k(\tau))^{q-1}\,z \,dz,\qquad h(u^{\pm}_{k(\tau)}):=\int\limits^{u^{\pm}_{k(\tau)}}_0(z\pm k(\tau))^{q-1} dz.$$

\end{lemma}
\begin{proof}
Consider  the case of sub-solution, using $\varphi= u^{+}_{k(t)} \zeta^p$ as a testing function in
the weak formulation, the contribution of the fractional diffusion term is estimated as in \cite[Proposition 2.1]{Lia2024}.
Concerning the parabolic part, formally we can write
$$u^{q-1}u_t u^{+}_{k(t)}\zeta^p=\frac{\partial}{\partial t}\Big(\int\limits^u_{k(t)}u^{q-1}(z-k(t))_+ dz \zeta^p\Big)
+k'(t)\int\limits^u_{k(t)}u^{q-1} dz\zeta^p-p\int\limits^u_{k(t)}u^{q-1}(z-k(t))_+dz \zeta^{p-1}\zeta_t.$$
Integrating this equality over the space-time cylinder and combining it with the estimate for the fractional diffusion term yields the desired conclusion. A rigorous treatment of the
time derivative can be adapted, for example, from \cite{BogDuzLia2021, BogDuzLiaSch2023}.
\end{proof}

\subsection{De Giorgi Type Lemma}
Construct the cylinder $\mathcal{Q}:= B_R\times (t_1, t_2) \subset \Omega_T$.
A standard consequence of the Energy estimates and Sobolev embedding theorem is the following De Giorgi type lemma.
\begin{lemma}\label{lemDG}
Let $u$ be a local weak super-solution to \eqref{eq1.1}-\eqref{eq1.3}, $u\geqslant 0$ in $\Omega_T$. Fix $k>0$ and let $\rho\in (0, \frac{1}{2}R)$, $\theta >0$ be the numbers such that $Q_{\rho, \theta}:=Q_{\rho, \theta}(x_0, t_0) \subset \mathcal{Q}$. There exist numbers $C>1$, $\nu \in (0, 1)$ depending only on the data and $r$, $\theta$, $k$, such that if
\begin{equation}\label{eqDG2}
\big|Q_{\rho, \theta}\cap \big\{u \leqslant k\big\}\big|\leqslant \nu |Q_{\rho, \theta}|,
\end{equation}
then either
\begin{equation}\label{eqDG3}
C\,\operatorname{Tail}(u_{-}, \mathcal{Q})\geqslant k,
\end{equation}
or
\begin{equation}\label{eqDG4}
u \geqslant \frac{1}{4}\,k,\quad \text{in}\quad Q_{\frac{1}{2}\rho, \frac{1}{2}\theta}.
\end{equation}
\end{lemma}
\begin{proof}
In what follows, we assume that inequality \eqref{eqDG3} does not hold, i.e.,
\begin{equation}\label{eqDG5}
C\,\operatorname{Tail}(u_{-}, \mathcal{Q})\leqslant \,k,
\end{equation}
where $C$ is a sufficiently large constant to be chosen later.

For $\sigma \in (0, 1]$ set
\[
\begin{aligned}
r_\sigma&:=\frac{\rho}{2}(1+\sigma),
&\qquad
\bar r_\sigma&:=\frac{r_\sigma+r_{\sigma/2}}{2},\\
\eta_\sigma&:=\frac{\theta}{2}(1+\sigma),
&
\bar\eta_\sigma&:=\frac{\eta_\sigma+\eta_{\sigma/2}}{2},
\end{aligned}
\]
and define the time-dependent levels
$$k_\sigma(t):=\frac{1}{2}k +\frac{1}{2}\sigma k -\tilde{C}^q\,\mathcal{K}(t),$$
where
$$\mathcal{K}(t):= k^{1-q}\,
\int\limits^{t}_{t_1}\,\,\int\limits_{\mathbb{R}^N\setminus B_R(x_0)}\frac{[u_-(x, \tau)]^{p-1}}{|x-x_0|^{N+sp}}dxd\tau.$$
We also set
\[
\begin{aligned}
B_\sigma&:=B_{r_\sigma}(x_0),
&\qquad
\bar B_\sigma&:=B_{\bar r_\sigma}(x_0),\\
Q_\sigma&:=B_\sigma\times(t_0-\eta_\sigma,t_0),
&
\bar Q_\sigma&:=\bar B_\sigma\times(t_0-\bar\eta_\sigma,t_0).
\end{aligned}
\]
Define the cutoff function $\zeta(x, t):=\zeta_1(x)\zeta_2(t)$, where $\zeta_1(x)\in C^1_0(\bar{B}_\sigma)$ and $\zeta_2(t)\in C^1(\mathbb{R})$ satisfy the following conditions:
$$0\leqslant \zeta_1(x)\leqslant 1, \quad \zeta_1(x)=1 \quad \text{in}\quad B_{\sigma/2},\quad \text{and}\quad |\nabla \zeta_1(x)|\leqslant\frac{2}{r_\sigma-r_{\sigma/2}},$$
$$
 0\leqslant \zeta_2(t)\leqslant 1,\quad \zeta_2(t)=
\begin{cases}
0, & t\leqslant t_0-\bar\eta_\sigma,\\
1, & t\geqslant t_0-\eta_{\sigma/2},
\end{cases} \quad
|\zeta'_2(t)|\leqslant \frac{2}{\eta_\sigma-\eta_{\sigma/2}}.
$$

 We  use inequality \eqref{eqenergy} for the function $u^-_{k_\sigma(t)}=(u-k_\sigma(t))_-$  and obtain
 \begin{equation}\label{eqDG6}
    \begin{aligned}
 &\sup\limits_{t_0-\bar{\eta}_{\sigma}\leqslant \tau \leqslant t_0}\int\limits_{\bar{B}_\sigma\times\{\tau\}}\int\limits_u^{k_\sigma(\tau)}
 z^{q-1}(k_\sigma(\tau)-z)\,dz\,\zeta^p\,dx\\&+\int\limits^{t_0}_{t_0-\bar{\eta}_\sigma}\int\limits_{\bar{B}_\sigma}\int\limits_{\bar{B}_\sigma}
 \frac{|u^-_{k_\sigma(\tau)}(x, \tau) -u^-_{k_\sigma(\tau)}(y, \tau)|^p}{|x-y|^{N+sp}}\min[\zeta(x, \tau), \zeta(y, \tau)]^p\,dxdy\,d\tau
 \\&\leqslant -\tilde{C}^q\,\iint\limits_{Q_\sigma}\mathcal{K}'(\tau)\,\int\limits_u^{k_\sigma(\tau)}z^{q-1}\,
 dz\,\zeta^p\,dxd\tau+ \gamma \iint\limits_{Q_\sigma}\int\limits_u^{k_\sigma(\tau)}
 z^{q-1}(k_\sigma(\tau)-z)\,dz\,\zeta^{p-1}\,|\zeta'_\tau|\,dxd\tau\\&+\gamma \iint\limits_{Q_\sigma}\int\limits_{B_\sigma}\big[u^-_{k_\sigma(\tau)}(x, \tau)+u^-_{k_\sigma(\tau)}(y, \tau)\big]^p\,\frac{|\zeta(x, \tau)-\zeta(y, \tau)|^p}{|x-y|^{N+sp}}dxdyd\tau\\&+\gamma \iint\limits_{Q_\sigma}
 u^-_{k_\sigma(\tau)}(x, \tau)\zeta^p(x, \tau)\bigg(\int\limits_{\mathbb{R}^N\setminus B_\sigma}\frac{\big[u^-_{k_\sigma(\tau)}(y, \tau)\big]^{p-1}}{|x-y|^{N+sp}} dy\bigg)dxd\tau=\mathcal{J}_1+\mathcal{J}_2+\mathcal{J}_3+\mathcal{J}_4.
 \end{aligned} 
 \end{equation}
We have in $\mathcal{Q}$
\begin{equation*}u^-_{k_\sigma(\tau)}\leqslant \gamma k,\quad\text{and}\quad
\int\limits_u^{k_\sigma(\tau)} z^{q-1}(k_\sigma(\tau)-z)\,dz\leqslant \gamma k^{q+1}.
\end{equation*}
so, we estimate $\mathcal{J}_2$ as follows
\begin{equation*}
\mathcal{J}_2\leqslant \frac{\gamma k^{q+1}}{\sigma \theta}\big|Q_\sigma\cap\big\{u\leqslant k_\sigma(\cdot)\big\}\big|.
\end{equation*}
To estimate $\mathcal{J}_3$, we use the Lipschitz continuity of the spatial cutoff function,
$$
|\zeta(x,\tau)-\zeta(y,\tau)|
\leqslant \frac{\gamma}{\sigma\rho}|x-y|.
$$
Consequently,
\begin{equation*}
  \mathcal{J}_3
  \leqslant \frac{\gamma k^p}{\sigma^p \rho^{ps}}\big|Q_\sigma\cap\big\{u\leqslant k_\sigma(\cdot)\big\}\big|.
\end{equation*}
Observe that condition \eqref{eqDG5} yields
$$k_\sigma(\tau)\geqslant \frac{1}{4}\,k,$$
provided that $\big(\frac{\tilde{C}}{C}\big)^q\leqslant \frac{1}{4}$. 
Using inequality \eqref{eqalg1} from the algebraic  Lemma \ref{lemAlg}, 
we obtain,  for $ u \leqslant k_\sigma(\tau)$,
\begin{equation}\label{eqDG7}
\int\limits_{u}^{k_\sigma(\tau)}z^{q-1} dz\geqslant \frac{1}{\gamma}(k_{\sigma}(\tau)-u)(k_\sigma(\tau)+u)^{q-1}\geqslant\frac{1}{\gamma}\,k^{q-1}\,u^-_{k_\sigma(\tau)},
\end{equation}
and hence,
\begin{equation*}
\begin{aligned}
\mathcal{J}_1&\leqslant -\tilde{C}^q\frac{k^{q-1}}{\gamma}\iint\limits_{Q_\sigma}u^-_{k_\sigma(\tau)}\,\mathcal{K}'(\tau)\zeta^p\,dxd\tau
\\&=-\frac{\tilde{C}^q}{\gamma}\iint\limits_{Q_\sigma}u^-_{k_\sigma(\tau)}(x, \tau)\,\zeta^p(x, \tau)\bigg(\int\limits_{\mathbb{R}^N\setminus B_R(x_0)}\frac{[u_-(y, \tau)]^{p-1}}{|y-x_0|^{N+sp}}\,dy\bigg)\,dxd\tau.
\end{aligned}
\end{equation*}
Finally, using the facts that
$$|y-x|\geqslant \frac{R-\bar{r}_\sigma}{R} |y-x_0|\geqslant \frac{1}{2}|y-x_0|,\quad x\in \bar{B}_\sigma,\quad y\in \mathbb{R}^N\setminus B_R,$$
and
$$|y-x| \geqslant |y-x_0|-|x_0-x|\geqslant r_\sigma-\bar{r}_\sigma=\frac{1}{8}\sigma\rho, \quad x\in \bar{B}_\sigma,\quad y\in \mathbb{R}^N\setminus B_\sigma,$$
we obtain
\begin{equation*}
\begin{aligned}
    &\mathcal{J}_4\leqslant\gamma \iint\limits_{Q_\sigma}
 u^-_{k_\sigma(\tau)}(x, \tau)\zeta^p(x, \tau)\bigg(\sup\limits_{x\in \bar{B}_\sigma}\int\limits_{\mathbb{R}^N\setminus B_R(x_0)}\frac{\big[u^-_{k_\sigma(\tau)}(y, \tau)\big]^{p-1}}{|x-y|^{N+sp}} dy\bigg)dxd\tau \\&+\gamma \iint\limits_{Q_\sigma}
 u^-_{k_\sigma(\tau)}(x, \tau)\zeta^p(x, \tau)\bigg(\sup\limits_{x\in \bar{B}_\sigma}\int\limits_{B_R(x_0)\setminus B_\sigma}\frac{\big[u^-_{k_\sigma(\tau)}(y, \tau)\big]^{p-1}}{|x-y|^{N+sp}} dy\bigg)dxd\tau \\&\leqslant
 \gamma \iint\limits_{Q_\sigma}
 u^-_{k_\sigma(\tau)}(x, \tau)\zeta^p(x, \tau)\bigg(\frac{k^{p-1}}{R^{sp}}+\int\limits_{\mathbb{R}^N\setminus B_R(x_0)}\frac{\big[u_-(y, \tau)\big]^{p-1}}{|y-x_0|^{N+sp}} dy\bigg)dxd\tau\\&+\frac{\gamma}{\sigma^{N+sp}}\iint\limits_{Q_\sigma}
 u^-_{k_\sigma(\tau)}(x, \tau)\zeta^p(x, \tau)\bigg(\int\limits_{B_R(x_0)\setminus B_\sigma}\frac{\big[u^-_{k_\sigma(\tau)}(y, \tau)\big]^{p-1}}{
 |y-x_0|^{N+sp}} dy\bigg)dxd\tau\\&\leqslant \gamma \iint\limits_{Q_\sigma}
 u^-_{k_\sigma(\tau)}(x, \tau)\zeta^p(x, \tau)\bigg(\int\limits_{\mathbb{R}^N\setminus B_R(x_0)}\frac{\big[u_-(y, \tau)\big]^{p-1}}{|y-x_0|^{N+sp}} dy\bigg)dxd\tau+\frac{\gamma k^p}{\sigma^{N+sp}\rho^{sp}}\big|Q_\sigma\cap \big\{u\leqslant k_\sigma(\cdot)\big\}\big|.
\end{aligned}
\end{equation*}
The first term on the right-hand side of this inequality  cancels with the   right-hand side of the estimate for $\mathcal{J}_1$, 
provided that $\tilde{C}$ is chosen to satisfy $\frac{\tilde{C}^q}{\gamma} = \gamma$.

Next we will estimate the left-hand side of \eqref{eqDG6}. By \eqref{eqDG7}, we have
\begin{equation*}
\begin{aligned}
\int\limits_{u}^{k_\sigma(\tau)}z^{q-1}(k_\sigma(\tau)-z) dz&\geqslant \int\limits_{u}^{\frac{1}{2}(k_\sigma(\tau)+u)}z^{q-1}(k_\sigma(\tau)-z)\,dz\\&\geqslant \frac{1}{2}(k_\sigma(\tau)-u)\,\int\limits_{u}^{\frac{1}{2}(k_\sigma(\tau)+u)}z^{q-1}\,dz
\geqslant \frac{1}{\gamma}k^{q-1}\,\big[u^-_{k_\sigma(\tau)}\big]^2,
\end{aligned}
\end{equation*}
which yields
\begin{equation*}
\sup\limits_{t_0-\bar{\eta}_{\sigma}\leqslant \tau \leqslant t_0}\int\limits_{\bar{B}_\sigma\times\{\tau\}}\int\limits_u^{k_\sigma(\tau)}
 z^{q-1}(k(\tau)-z)\,dz\,\zeta^p\,dx\geqslant \frac{k^{q-1}}{\gamma} \sup\limits_{t_0-\bar{\eta}_{\sigma}\leqslant \tau \leqslant t_0}\,\,\int\limits_{\bar{B}_\sigma\times\{\tau\}}\big[u^-_{k_\sigma(\tau)}\big]^2\,\zeta^p\,dx.
\end{equation*}
Combining the last inequality with the previous estimates, we can rewrite \eqref{eqDG6} as
\begin{equation}\label{eqDG8}
\begin{aligned}
&k^{q-1}\,\sup\limits_{t_0-\bar{\eta}_{\sigma}\leqslant \tau \leqslant t_0}\,\,\int\limits_{\bar{B}_\sigma\times\{\tau\}}\big[u^-_{k_\sigma(\tau)}\big]^2\,\zeta^p\,dx
\\&+
\int\limits^{t_0}_{t_0-\bar{\eta}_\sigma}\int\limits_{\bar{B}_\sigma}\int\limits_{\bar{B}_\sigma}
 \frac{|u^-_{k_\sigma(\tau)}(x, \tau) -u^-_{k_\sigma(\tau)}(y, \tau)|^p}{|x-y|^{N+sp}}\min[\zeta(x, \tau), \zeta(y, \tau)]^p\,dxdy\,d\tau
 \\&\leqslant \frac{\gamma k^p}{\sigma^\gamma \rho^{sp}}\Big[1+ k^{q+1-p}\frac{\rho^{sp}}{\theta}\Big]
\big|Q_\sigma\cap \big\{u\leqslant k_\sigma(\cdot)\big\}\big|. 
\end{aligned}
\end{equation}
Now, using the Sobolev embedding theorem, see Lemma \ref{lemEmbedding}, we obtain 
\begin{equation*}
\begin{aligned}
&\big(\frac{\sigma}{4}k)^{p(1+\frac{2s}{N})}\big|Q_{\sigma/2}\cap \big\{u\leqslant k_{\sigma/2}(\cdot)\big\}\big|\leqslant
\iint\limits_{\bar{Q}_\sigma}\big[u^-_{k_\sigma(\tau)}(x, \tau)\big]^{p(1+\frac{2s}{N})}[\zeta(x, \tau)]^{p(1+\frac{sp}{N})}\,dxd\tau 
\\&\leqslant \gamma \bigg(\sup\limits_{t_0-\bar{\eta}_{\sigma}\leqslant \tau \leqslant t_0}\int\limits_{\bar{B}_\sigma\times\{\tau\}}\big[u^-_{k_\sigma(\tau)}\big]^2\,\zeta^p\,dx\bigg)^{\frac{sp}{N}}
\\&\times
\bigg(\int\limits^{t_0}_{t_0-\bar{\eta}_\sigma}\int\limits_{\bar{B}_\sigma}\int\limits_{\bar{B}_\sigma}
 \frac{|u^-_{k_\sigma(\tau)}(x, \tau)\zeta(x, \tau) -u^-_{k_\sigma(\tau)}(y, \tau)\zeta(y,\tau)|^p}{|x-y|^{N+sp}}\,dxdy\,d\tau
+\frac{\gamma}{\sigma^\gamma\rho^{sp}}\iint\limits_{\bar{Q}_\sigma}[u^-_{k_\sigma(\tau)}]^p dxd\tau \bigg).
\end{aligned}
\end{equation*}
To estimate the first term in the last line above, we use
\begin{multline*}
|u^-_{k_\sigma(\tau)}(x, \tau)\zeta(x, \tau) -u^-_{k_\sigma(\tau)}(y, \tau)\zeta(y,\tau)|^p\leqslant\\\leqslant
\gamma|u^-_{k_\sigma(\tau)}(x, \tau) -u^-_{k_\sigma(\tau)}(y, \tau)|^p\zeta^p(x, \tau)+\gamma |\zeta(x, \tau) -\zeta(y,\tau)|^p[u^-_{k_\sigma(\tau)}(y, \tau)]^p.
\end{multline*}
Substituting this estimate into the preceding inequality, and then applying \eqref{eqDG8} together with the estimate for $\mathcal{J}_3$, and setting
\[
\mathbf{Y}(\sigma):=
\frac{\big|Q_{\sigma}\cap \big\{u\leqslant k_{\sigma}(\cdot)\big\}\big|}
{|Q_{\sigma}|},
\]
we arrive at
\begin{equation*}
\mathbf{Y}(\sigma/2)\leqslant \frac{\gamma}{\sigma^\gamma}
\Big(k^{p-1-q}\frac{\theta}{\rho^{sp}}\Big)^{\frac{sp}{N}}\Big[1+k^{q+1-p}\frac{\rho^{sp}}{\theta}\Big]^{1+\frac{sp}{N}}
\big[\mathbf{Y}(\sigma)\big]^{1+\frac{sp}{N}}.
\end{equation*}
To conclude the iteration, let 
$\sigma_j:=2^{1-j},\,\, Y_j:=\mathbf{Y}(\sigma_j),  j=1,2,\ldots$ Then $\sigma_{j+1}=\sigma_j/2$ and the previous inequality yields
\begin{equation*}
Y_{j+1}\leqslant \gamma 2^{j\gamma}
\Big(k^{p-1-q}\frac{\theta}{\rho^{sp}}\Big)^{\frac{sp}{N}}\Big[1+k^{q+1-p}\frac{\rho^{sp}}{\theta}\Big]^{1+\frac{sp}{N}}
Y_j^{1+\frac{sp}{N}}.
\end{equation*}
Thus, by the standard fast geometric convergence lemma, cf. \cite[Chapter I, Lemma 4.1]{DiB1993}, $\lim\limits_{j \rightarrow \infty}Y_j=0,$  provided that
$$Y_1=\mathbf{Y}(1)\leqslant \nu:=\frac{1}{\gamma}k^{q+1-p}\frac{\rho^{sp}}{\theta}\Big[1+k^{q+1-p}\frac{\rho^{sp}}{\theta}\Big]^{-1-\frac{N}{sp}},$$
which holds by \eqref{eqDG2} and by the fact that
$$\big|Q_{\rho, \theta}\cap\big\{u\leqslant k_1(\cdot)\big\}\big|\leqslant \big|Q_{\rho, \theta}\cap\big\{u\leqslant k\big\}\big|\leqslant \nu\,|Q_{\rho, \theta}|.$$
Choosing $C=4^{\frac{1}{q}}\tilde{C}$, we arrive at the required \eqref{eqDG4}, this completes the proof of the lemma.
\end{proof}

\subsection{De Giorgi-type Lemma Involving the Initial Data}

\begin{lemma}\label{lemDGdata1}
Let $q>0$ and $p>1$, let $u$ be a  local weak super-solution to \eqref{eq1.1}-\eqref{eq1.3}, $u \geqslant 0$ in $\Omega_T$. Fix $k > 0$ and let $\rho\in (0, \frac{1}{2}R)$.  There exist numbers $C>1$, $\nu_0 >0$  depending only on the data  such that if
\begin{equation}\label{eqDGdt2}
u(x, t_0)\geqslant k,\quad x\in B_\rho(x_0),
\end{equation}
then either
\begin{equation}\label{eqDGdt3}
C\,\operatorname{Tail}(u_-, \mathcal{Q})\geqslant \,k,
\end{equation}
or
\begin{equation}\label{eqDGdt4}
u \geqslant \frac{1}{4}\,k,\quad \text{in}\quad B_{\frac{\rho}{2}}(x_0)\times (t_0, t_0+ \nu_0\,\theta),\quad \theta:= k^{q-p+1}\,\rho^{sp},
\end{equation}
provided that $B_{4\rho}(x_0)\times(t_0, t_0+\nu_0\,\theta) \subset \mathcal{Q}.$
\end{lemma}
\begin{proof}
Let 
$$k_\sigma(t)=\frac{1}{2} k+\frac{1}{2}\sigma k- \tilde{C}^q \mathcal{K}(t),$$
where $\mathcal{K}(t)$ is the same as in the previous lemma.   
 We retain the notation $r_\sigma, \bar r_\sigma, B_\sigma$, and $\bar B_\sigma$ introduced in the proof of Lemma~\ref{lemDG}, and let $\zeta_1(x)$ be the corresponding spatial cutoff function. We apply \eqref{eqenergy} to the function $u^-_{k(t)}$ in
$ B_\sigma\times(t_0,t_0+\nu_0\theta).$  Since $k(t_0)\leqslant k$, assumption \eqref{eqDGdt2} implies
$$u(x,t_0)\geqslant k\geqslant k(t_0),$$
and hence $u^-_{k(t_0)}(x,t_0)=0$. Thus, the spatial integral at $t=t_0$  (i.e., the term at the time level $t_0-\eta$ in Lemma \ref{lemEnergy}) vanishes. Moreover, since $\zeta_1$ is independent of time, the corresponding time-derivative term is zero as well. Arguing as in the derivation of \eqref{eqDG8}, we obtain 
\begin{equation}\label{eqDGdt5}
\begin{aligned}
k^{q-1}\,\sup\limits_{t_0\leqslant \tau \leqslant t_0+\nu_0\theta}\,\,\,&\int\limits_{\bar{B}_\sigma\times\{\tau\}}\big[u^-_{k_\sigma(\tau)}\big]^2\,\zeta^p_1(x)\,dx\\&+
\int\limits^{t_0+\nu_0\theta}_{t_0}\int\limits_{\bar{B}_\sigma}\int\limits_{\bar{B}_\sigma}
 \frac{|u^-_{k_\sigma(\tau)}(x, \tau) -u^-_{k_\sigma(\tau)}(y, \tau)|^p}{|x-y|^{N+sp}}\min[\zeta_1(x), \zeta_1(y)]^p\,dxdy\,d\tau
 \\&\leqslant \frac{\gamma k^p}{\sigma^\gamma \rho^{sp}}
\big|B_\sigma\times(t_0, t_0+\nu_0\theta)\cap \big\{u\leqslant k_\sigma(\cdot)\big\}\big|. 
\end{aligned}
\end{equation}
By the Sobolev embedding, see Lemma~\ref{lemEmbedding}, and proceeding as in the proof of Lemma~\ref{lemDG}, denoting
$$\mathbf{Y}(\sigma):=\frac{\big|B_{\sigma}\times(t_0, t_0+\nu_0\theta)\cap \big\{u\leqslant k_{\sigma}(\cdot)\big\}\big|}{|B_{\sigma}|\,\nu_0\theta},$$
we obtain
\begin{equation*}
\mathbf{Y}(\sigma/2)\leqslant \frac{\gamma}{\sigma^\gamma}
\nu_0^{\frac{sp}{N}}
\big[\mathbf{Y}(\sigma)\big]^{1+\frac{sp}{N}}.
\end{equation*}
Repeating the iteration from Lemma~\ref{lemDG}, with $Y_j:=\mathbf{Y}(\sigma_j)$, we conclude that
$$\lim_{j\to\infty}Y_j=0,$$
provided that
$$Y_1=\mathbf{Y}(1)\leqslant\frac{1}{\gamma\nu_0}.$$
 Choosing $\nu_0=\frac{1}{\gamma}$ and recalling the choice $C=4^{\frac1q}\widetilde C$ made in Lemma~\ref{lemDG}, we complete the proof of the lemma.
\end{proof}

\subsection{De Giorgi-type Lemma Under an Initial Measure Condition}

The following lemma incorporates the proofs of both, a  De Giorgi lemma and its version involving the initial data.

\begin{lemma}\label{lemDGmeasuredata}
Let $q>0$ and $p>1$, let $u$ be a  local weak super-solution to \eqref{eq1.1}-\eqref{eq1.3}, $u \geqslant 0$ in $\Omega_T$. Fix $k> 0$ and let $\rho\in (0, \frac{1}{2}R)$.  There exist numbers $C>1$, $\delta_1$, $\nu_1 \in (0, 1)$  depending only on the data,  such that if
\begin{equation}\label{eqDGmdt2}
\big|B_\rho(x_0)\cap\big\{u(\cdot, t_0)\leqslant k\big\}\big|
\leqslant \nu_1|B_\rho(x_0)|,
\end{equation}
then either
\begin{equation}\label{eqDGmdt3}
C\,\operatorname{Tail}(u_{-}, \mathcal{Q})\geqslant \,k,
\end{equation}
or
\begin{equation}\label{eqDGmdt4}
u \geqslant \frac{1}{4}\,k,\quad \text{in}\quad B_{\frac{\rho}{2}}(x_0)\times (t_0+\frac{1}{2}\delta_1\theta, t_0+ \delta_1\,\theta),\quad \theta:= k^{q-p+1}\,\rho^{sp},
\end{equation}
provided that $B_{4\rho}(x_0)\times(t_0, t_0+\delta_1\,\theta) \subset \mathcal{Q}.$
\end{lemma}
\begin{proof}
Set $$k_\sigma(t)=\frac{1}{2} k+\frac{1}{2}\sigma k- \tilde{C}^q \mathcal{K}(t),$$ 
where $\mathcal{K}(t)$  is defined  as in Lemma \ref{lemDG},  and let $\zeta_1(x)$ be the cutoff function from that lemma. We use the same notation as in the proof of Lemma~\ref{lemDG}, with $\theta$ replaced by $\delta_1\theta$ in the definitions of $\eta_\sigma$ and $\bar\eta_\sigma$, and with $t_0$ replaced by $t_0+\delta_1\theta$ in the definitions of $Q_\sigma$ and $\bar Q_\sigma$. We apply the energy estimate \eqref{eqenergy} to $u^-_{k_\sigma(t)}$ with the cutoff $\zeta_1$. Since $\zeta_1$ is independent of time, the corresponding time-derivative term vanishes. 
Using \eqref{eqDGmdt2} and arguing as in the derivation of \eqref{eqDG8}, we obtain
\begin{equation*}
\begin{aligned}
k^{q-1}\,&\sup\limits_{t_0+\delta_1\theta-\bar{\eta}_\sigma \leqslant \tau \leqslant t_0+\delta_1\theta}\int\limits_{\bar{B}_\sigma\times\{\tau\}}\big[u^-_{k_\sigma(\tau)}\big]^2\,\zeta^p_1(x)\,dx\\&+
\int\limits^{t_0+\delta_1\theta}_{t_0+\delta_1\theta-\bar{\eta}_\sigma }\int\limits_{\bar{B}_\sigma}\int\limits_{\bar{B}_\sigma}
 \frac{|u^-_{k_\sigma(\tau)}(x, \tau) -u^-_{k_\sigma(\tau)}(y, \tau)|^p}{|x-y|^{N+sp}}\min[\zeta_1(x), \zeta_1(y)]^p\,dxdy\,d\tau
 \\&\leqslant \frac{\gamma k^p}{\sigma^\gamma \rho^{sp}}
\big|Q_\sigma\cap \big\{u\leqslant k_\sigma(\cdot)\big\}\big|+\gamma \,\nu_1 k^{q+1}\,|B_\sigma|\leqslant
 \frac{\gamma k^p}{\sigma^\gamma \rho^{sp}}\,\bigg[\big|Q_\sigma\cap \big\{u\leqslant k_\sigma(\cdot)\big\}\big|+\frac{\nu_1}{\delta_1}|Q_\sigma|\bigg].
\end{aligned}
\end{equation*}
Suppose first that, for some $\sigma_0\in(0,1)$,
$$\big|Q_{\sigma_0}\cap \big\{u\leqslant k_{\sigma_0}(\cdot)\big\}\big|\leqslant \frac{\nu_1}{\delta_1}|Q_{\sigma_0}|.$$
Let $\nu\in(0,1)$ be the constant from Lemma~\ref{lemDG}, and choose $\nu_1$ so that $\nu_1\leqslant\delta_1\nu$. Then the hypothesis of Lemma~\ref{lemDG} is satisfied. Hence, either \eqref{eqDGmdt3} holds, or we obtain the lower bound \eqref{eqDGmdt4}. Assume now that, for any $\sigma\in(0,1)$,
$$\big|Q_{\sigma}\cap \big\{u\leqslant k_{\sigma}(\cdot)\big\}\big|\geqslant \frac{\nu_1}{\delta_1}|Q_\sigma|,$$
then
\begin{equation*}
\begin{aligned}
k^{q-1}\,\sup\limits_{t_0+\delta_1\theta-\bar{\eta}_\sigma \leqslant \tau \leqslant t_0+\delta_1\theta}&\int\limits_{\bar{B}_\sigma\times\{\tau\}}\big[u^-_{k_\sigma(\tau)}\big]^2\,\zeta^p_1(x)\,dx\\&+
\int\limits^{t_0+\delta_1\theta}_{t_0+\delta_1\theta-\bar{\eta}_\sigma }\int\limits_{\bar{B}_\sigma}\int\limits_{\bar{B}_\sigma}
 \frac{|u^-_{k_\sigma(\tau)}(x, \tau) -u^-_{k_\sigma(\tau)}(y, \tau)|^p}{|x-y|^{N+sp}}\min[\zeta_1(x), \zeta_1(y)]^p\,dxdy\,d\tau
\\&\leqslant 
 \frac{\gamma k^p}{\sigma^\gamma \rho^{sp}}\,\big|Q_\sigma\cap \big\{u\leqslant k_\sigma(\cdot)\big\}\big|. 
\end{aligned}
\end{equation*}
By Lemma~\ref{lemEmbedding} and proceeding as in Lemma~\ref{lemDG}, we set
\begin{equation*}
\mathbf{Y}(\sigma):=\frac{\big|Q_{\sigma}\cap \big\{u\leqslant k_{\sigma}(\cdot)\big\}\big|}{|Q_{\sigma}|},
\end{equation*}
and arrive at
\begin{equation*}
\mathbf{Y}(\sigma/2)\leqslant \frac{\gamma}{\sigma^\gamma}
\delta_1^{\frac{sp}{N}}
\big[\mathbf{Y}(\sigma)\big]^{1+\frac{sp}{N}}.
\end{equation*}
Repeating the iteration from Lemma~\ref{lemDG}, with
$\sigma_j:=2^{1-j},\,\, Y_j:=\mathbf{Y}(\sigma_j),  j=1,2,\ldots$, we obtain $Y_j\to0$ as $j\to\infty$, provided that
$$\mathbf{Y}(1)\leqslant \frac{1}{\gamma \delta_1}.$$
Choosing $\delta_1$  and recalling the choice $C=4^{\frac1q}\widetilde C$ made in Lemma~\ref{lemDG}, we complete the proof of the lemma.
\end{proof}

\section{Expansion of Positivity}\label{sec:ExpPos}

\subsection{Propagation of Positivity in Measure}

The following lemma propagates measure theoretical information forward in time.

\begin{lemma}\label{lemPropagation}
Let $q>0$ and $p>1$, and let $u$ be a  local weak super-solution to 
\eqref{eq1.1}-\eqref{eq1.3}, $u \geqslant 0$ in $\Omega_T$. 
Suppose  that
\begin{equation}\label{eqProp2}
\big|B_r(x_0)\cap\big\{u(\cdot, t_0)\geqslant k \big\}\big|\geqslant \alpha |B_r(x_0)|,
\end{equation}
for some $k>0$ and $\alpha\in(0,1)$. Then there exist constants $\tilde{\epsilon}$, $\tilde{\delta} \in (0, 1)$ depending
only on the data and $\alpha$, such that
either
\begin{equation}\label{eqProp3}
\,\operatorname{Tail}(u_{-}, \mathcal{Q}) > \tilde{\delta}^{\frac{1}{q}}\,k,
\end{equation}
or
\begin{equation}\label{eqProp4}
\big|B_r(x_0)\cap\big\{u(\cdot, t))\geqslant \tilde{\epsilon}\,k\big\}\big|\geqslant \frac{1}{4}\alpha |B_r(x_0)|,
\quad\text{for all}\quad  t\in(t_0, t_0+ \tilde{\delta}\,k^{q-p+1}\,r^{sp}), 
\end{equation}
provided that  $B_{4r}(x_0)\times (t_0, t_0+ k^{q-p+1}\,r^{sp})\subset \mathcal{Q}$.

\end{lemma}

\begin{proof}
Set $I:= (t_0, t_0+ \tilde{\delta}\,k^{q-p+1}\,r^{sp})$. Use the energy estimate \eqref{eqenergy} for the function $u^-_k$  in the cylinder $B_{r}(x_0)\times I$,  with
 the cutoff function $\zeta(x)\in C^1_0(B_{r(1-\sigma/2)}(x_0))$ such that 
 \[
0\leqslant\zeta\leqslant1,
\qquad
\zeta\equiv1
\quad\text{in }B_{r(1-\sigma)}(x_0),
\qquad
|\nabla\zeta|\leqslant\frac{2}{\sigma r}.
\]
  Estimating the diffusion term in the standard way,  we obtain
\begin{multline*}
\sup\limits_{t\in I}\int\limits_{B_{r(1-\sigma)}(x_0)\times\{t\}}\int\limits_u^{k}z^{q-1}(k-z)\,dzdx\leqslant 
\int\limits_{B_{r}(x_0)\times\{t_0\}}\int\limits_u^{k}z^{q-1}(k-z)\,dzdx +\frac{\gamma\,\tilde{\delta}}{\sigma^p}\,k^{q+1} |B_r(x_0)|
\\+\gamma \int\limits_I\int\limits_{B_{r(1-\sigma/2)}(x_0)}(u(x,t)-k)_-\bigg(\sup\limits_{x\in B_{r(1-\sigma/2)}(x_0)}\int\limits_{\mathbb{R}^N\setminus B_r(x_0)}\frac{\big[(u(y, t)-k)_-\big]^{p-1}}{|x-y|^{N+sp}} dy\bigg) dx dt.
\end{multline*}
Assuming that \eqref{eqProp3} does not hold, we estimate the last term as follows
\begin{equation}
\begin{aligned}
\int\limits_I\,\,\int\limits_{B_{r(1-\sigma/2)}(x_0)}&(u(x,t)-k)_-\bigg(\sup\limits_{x\in B_{r(1-\sigma/2)}(x_0)}\int\limits_{\mathbb{R}^N\setminus B_r(x_0)}\frac{\big[(u(y, t)-k)_-\big]^{p-1}}{|x-y|^{N+sp}} dy\bigg) dx dt
\\&\leqslant \frac{\gamma}{\sigma^{N+sp}}\int\limits_I\int\limits_{B_{r}(x_0)}(u(x,t)-k)_-\bigg(\int\limits_{B_R(x_0)\setminus B_r(x_0)}\frac{\big[(u(y, t)-k)_-\big]^{p-1}}{|y-x_0|^{N+sp}} dy\bigg) dx dt\\&+
\gamma\int\limits_I\int\limits_{B_{r}(x_0)}(u(x,t)-k)_-\bigg(\int\limits_{\mathbb{R}^N\setminus B_R(x_0)}\frac{\big[(u(y, t)-k)_-\big]^{p-1}}{|y-x_0|^{N+sp}} dy\bigg) dx dt
\\&\leqslant \frac{\gamma\,\tilde{\delta}}{\sigma^{N+sp}}\,k^{q+1} |B_r(x_0)|+\gamma k \big[\operatorname{Tail} (u_-, \mathcal{Q})\big]^q|B_r(x_0)|\leqslant\frac{\gamma\,\tilde{\delta}}{\sigma^{N+sp}}\,k^{q+1} |B_r(x_0)|.
\end{aligned}
\end{equation}
Combining the two preceding inequalities, we derive
\begin{equation}\label{eqProp5}
\sup\limits_{t\in I}\int\limits_{B_{r(1-\sigma)}(x_0)\times\{t\}}\int\limits_u^{k}z^{q-1}(k-z)\,dzdx\leqslant 
\int\limits_{B_{r}(x_0)\times\{t_0\}}\int\limits_u^{k}z^{q-1}(k-z)\,dzdx +\frac{\gamma\,\tilde{\delta}}{\sigma^{N+p}}\,k^{q+1} |B_r(x_0)|.
\end{equation}
To complete the proof, it remains to estimate the integrals on the left- and right-hand sides of \eqref{eqProp5}. Set $k_\epsilon:=\epsilon k$. Then \eqref{eqProp5} and \eqref{eqProp2} yield, for every $t\in I$,
\begin{multline*}
\big|B_{r(1-\sigma)}(x_0)\cap\big\{u(\cdot, t)\leqslant k_\epsilon\big\}\big|\leqslant \frac{\int\limits^k_0 z^{q-1}(k-z) dz}{\int\limits^k_{k_\epsilon} z^{q-1}(k-z) dz}(1-\alpha) |B_r(x_0)|+\frac{\gamma\,\sigma^{-N-p} \tilde{\delta}}{\int\limits^k_{k_\epsilon}z^{q-1}(k-z) dz}\,k^{q+1} |B_r(x_0)|.
\end{multline*}
Moreover,  on the set $B_{r(1-\sigma)}(x_0)\cap\{u(\cdot, t)\leqslant k_\epsilon\}$, we have
$$\int\limits_u^{k}z^{q-1}(k-z)\,dz\geqslant \int\limits_{k_\epsilon}^{k}z^{q-1}(k-z)\,dz\geqslant \frac{k^{q+1}}{q(q+1)}(1-(q+1)\epsilon^q),$$
and
\begin{equation*}
\begin{aligned}
\int\limits_u^{k}z^{q-1}(k-z)\,dz\leqslant \int\limits^{k}_{0}z^{q-1}(k-z)\,dz&=\int\limits_{k_\epsilon}^{k}z^{q-1}(k-z)\,dz+
\int\limits_{0}^{k_\epsilon}z^{q-1}(k-z)\,dz\\&\leqslant \int\limits_{k_\epsilon}^{k}z^{q-1}(k-z)\,dz+\epsilon^q\frac{k^{q+1}}{q}.
\end{aligned}
\end{equation*}
Hence
\begin{equation*}
\begin{aligned}
\big|B_{r(1-\sigma)}(x_0)\cap\big\{u(\cdot, t)\leqslant k_\epsilon\big\}\big|&\leqslant \Bigg[1+\frac{\int\limits_{0}^{k_\epsilon}z^{q-1}(k-z)\,dz}{\int\limits_{k_\epsilon}^{k}z^{q-1}(k-z)\,dz}\Bigg](1-\alpha)|B_r(x_0)|+
\frac{\gamma\,\tilde{\delta}}{\sigma^{N+p}} |B_r(x_0)|\\&\leqslant
\Big[(1+\gamma(q)\epsilon^q)(1-\alpha)+\frac{\,\tilde{\delta}}{\sigma^{N+p}}\Big] |B_r(x_0)|,
\end{aligned}
\end{equation*}
from which we conclude
\begin{equation*}
\big|B_{r}(x_0)\cap\big\{u(\cdot, t)\leqslant k_\epsilon\big\}\big|\leqslant \Big[N\sigma+ (1+\gamma(q)\epsilon^q)(1-\alpha)+\frac{\gamma\,\tilde{\delta}}{\sigma^{N+p}}\Big]|B_r(x_0)|.
\end{equation*}
Choose $\sigma$, $\epsilon$ small enough such that  $N\sigma=\frac{1}{8}\alpha$, $(1+\gamma(q)\epsilon^q)(1-\alpha)=1-\frac{\alpha}{2}$ and then fixing $\tilde{\delta}$ by $\frac{\gamma\,\tilde{\delta}}{\sigma^{N+p}}=\frac{1}{8}\alpha$, we arrive at the required \eqref{eqProp4}, which completes the proof of the lemma.
\end{proof}

\subsection{Two Versions of the Measure Shrinking Lemma}
The following measure shrinking lemmas can  be readily derived.  The argument runs  like in \cite[Lemma 3.4]{Lia2024} and \cite[Lemma 3.6, Lemma 3.7]{Lia2026arxiv}.
\begin{lemma}\label{lemShrinking}
Let $u$ be a  local weak super-solution to \eqref{eq1.1}-\eqref{eq1.3}, $u \geqslant 0$ in $\Omega_T$,   with $q>0$ and $p>1$.  Suppose that, for some $\delta\in(0,1)$,
$k>0$, and $\alpha\in(0,1)$,
\begin{equation}\label{eqShr2}
\big|B_{\frac{3}{4} r}(x_0)\cap\big\{u(\cdot, t))\geqslant k\big\}\big|\geqslant \alpha |B_{\frac{3}{4} r}(x_0)|,
\end{equation}
for all $t\in (t_0-\delta k^{q-p+1} r^{sp}, t_0)$. Then there exists $\gamma>0$, depending only on the data, such that
for every $\epsilon\in(0,\frac12)$, either  
\begin{equation}\label{eqShr3}
\operatorname{Tail}(u_{-}, \mathcal{Q}) > \epsilon \,k,
\end{equation}
or there exists a time level $\tilde{t}\in (t_0-\delta k^{q-p+1} r^{sp}, t_0)$ such that
\begin{equation}\label{eqShr4}
\big|B_{\frac{3}{4} r}(x_0)\cap\big\{u(\cdot, \tilde{t})\leqslant \frac{1}{2}\epsilon k\big\}\big|\leqslant \frac{\gamma\,\epsilon^{q}}{\delta\,\alpha}|B_{\frac{3}{4} r}(x_0)|,
\end{equation}
provided that $Q^-_{4 r, \eta}(x_0, t_0)\subset \mathcal{Q}$, $\eta:=k^{q-p+1} r^{sp}$.
\end{lemma}
\begin{proof}
We apply the energy estimate \eqref{eqenergy} to the function
$u^-_{k_\epsilon}$, where $k_\epsilon=\epsilon k$.
Let $\zeta\in C^1_0(B_{\frac78 r}(x_0))$ be a cutoff function satisfying
\[
0\leqslant\zeta\leqslant1,
\qquad
\zeta\equiv1
\quad\text{in }B_{\frac34 r}(x_0),
\qquad
|\nabla\zeta|\leqslant\frac{8}{r}.
\]
We obtain 
\begin{equation*}
\begin{aligned}
\int\limits^{t_0}_{t_0-\delta \eta}\,\,\int\limits_{B_{\frac{3}{4} r}(x_0)}u^-_{k_\epsilon}(x, \tau)&\,\bigg(\int\limits_{B_{\frac{3}{4}r}(x_0)}\frac{\big[u^+_{k_\epsilon}(y, \tau)\big]^{p-1}}{|x-y|^{N+sp}}dy\bigg)\,dxd\tau\\&\leqslant \gamma \int\limits_{B_{r}(x_0)\times\{t_0-\delta \eta\}}\,\,\int\limits_u^{k_\epsilon}\,z^{q-1}(k_\epsilon-z) dz  dx\\&+
\gamma\int\limits^{t_0}_{t_0-\delta \eta}\,\,\int\limits_{B_r(x_0)}\int\limits_{B_r(x_0)}\max[u^-_{k_\epsilon}(x, \tau), u^-_{k_\epsilon}(y, \tau)]^p\frac{|\zeta(x)-\zeta(y)|^p}{|x-y|^{N+sp}}\,dxdyd\tau\\&+
\int\limits^{t_0}_{t_0-\delta \eta}\,\,\int\limits_{B_{\frac{7}{8}r}(x_0)}u^-_{k_\epsilon}(x, \tau)\bigg(\sup\limits_{x\in B_{\frac{7}{8} r}(x_0)}\int\limits_{\mathbb{R}^N\setminus B_r(x_0)}\frac{\big[ u^-_{k_\epsilon}(y, \tau)\big]^{p-1}}{|x-y|^{N+sp}} dy\bigg) dxd\tau.
\end{aligned}
\end{equation*}
Estimating the integrals on the right-hand side similarly to that of Lemma \ref{lemPropagation}, we get
\begin{multline*}
\int\limits^{t_0}_{t_0-\delta \eta}\,\,\int\limits_{B_{\frac{3}{4} r}(x_0)}u^-_{k_\epsilon}(x, \tau)\,\bigg(\int\limits_{B_{\frac{3}{4}r}(x_0)}\frac{\big[u^+_{k_\epsilon}(y, \tau)\big]^{p-1}}{|x-y|^{N+sp}}dy\bigg)\,dxd\tau\leqslant \gamma (\epsilon k)^{q+1}|B_r(x_0)|\\+ \frac{\gamma (\epsilon k)^p}{r^{sp}} |Q^-_{r, \delta\eta}(x_0, t_0)|\leqslant \frac{\gamma \epsilon^{q+1} k^p}{\delta r^{sp}}|Q^-_{r, \delta\eta}(x_0, t_0)|,
\end{multline*}
provided that \eqref{eqShr3} does not hold.
To estimate the left-hand side from below, observe that, for $x\in Q^-_{\frac{3}{4} r, \delta \eta}(x_0, t_0)\cap\big\{u\leqslant \frac{1}{2} \epsilon\,k\}$, we have
\begin{equation*}
u^-_{k_\epsilon}(x, \tau)\geqslant \frac{1}{2} \epsilon\,k.
\end{equation*}
On the other hand, for $y\in Q^-_{\frac{3}{4} r, \delta \eta}(x_0, t_0)\cap\big\{u\geqslant \,k\},$ it follows that
\begin{equation*}
u^+_{k_\epsilon}(x, \tau)\geqslant (1- \epsilon)k \geqslant \frac{1}{2} k,
\end{equation*}
provided that $\epsilon\in(0, \frac{1}{2}).$ Moreover, since $x,y\in B_{\frac{3}{4}r}(x_0)$, we have $|x-y|\leqslant\frac{3}{4}r$. Hence, by \eqref{eqShr2},
\begin{equation*}
\alpha\,\epsilon k^p\big|Q^-_{\frac{3}{4} r, \delta \eta}(x_0, t_0)\cap\big\{u\leqslant \frac{1}{2} \epsilon\,k\big\}\big|\leqslant \frac{\gamma}{\delta}\epsilon^{q+1} k^p |Q^-_{\frac{3}{4} r, \delta\eta}(x_0, t_0)|,
\end{equation*}
which yields
$$
\big|Q^-_{\frac{3}{4} r, \delta \eta}(x_0, t_0)\cap\big\{u\leqslant \frac{1}{2} \epsilon\,k\big\}\big|\leqslant \frac{\gamma \epsilon^q}{\delta \alpha} |Q^-_{\frac{3}{4} r, \delta\eta}(x_0, t_0)|.
$$
From which the required \eqref{eqShr4} follows. This completes the proof of the lemma.
\end{proof}

\begin{lemma}\label{lemShrinking1}
Let $u$ be a  local weak super-solution to \eqref{eq1.1}-\eqref{eq1.3},   with $u\geqslant0$ in $\Omega_T$, $q>0$, and $p>1$. Fix $k>0$ and $\delta\in(0,1)$.
 For every $\epsilon \in (0, \frac{1}{2})$ there exists $\gamma >0$, depending only on the data, such that either
\begin{equation}\label{eqShr6}
\operatorname{Tail}(u_{-}, \mathcal{Q}) > \epsilon\,k,
\end{equation}
or there exists a time level $\tilde{t}\in (t_0-\delta k^{q-p+1} r^{sp}, t_0)$ such that
\begin{equation}\label{eqShr7}
\big|B_{\frac{3}{4} r}(x_0)\cap\big\{u(\cdot, \tilde{t})\leqslant \frac{1}{2} \epsilon k\big\}\big|\leqslant \frac{\gamma (\epsilon k)^q}{\int\limits^{t_0}_{t_0-\delta \eta}\int\limits_{\mathbb{R}^N}\frac{[u_+(y, \tau)]^{p-1}}{(r+|x_0-y|)^{N+sp}}dy d\tau}|B_r(x_0)|,
\end{equation}
provided that $Q^-_{4 r, \eta}(x_0, t_0)\subset \mathcal{Q}$, $\eta:=k^{q-p+1} r^{sp}$.
\end{lemma}

\begin{proof}
We use \eqref{eqenergy} with $k_\epsilon =\epsilon k$ and choose the cutoff function as in the previous lemma. Estimating the terms on the right-hand side in the same way as in the proof of Lemma~\ref{lemShrinking}, and assuming that \eqref{eqShr6} does not hold, we obtain
\begin{equation*}
\begin{aligned}
\int\limits^{t_0}_{t_0-\delta \eta}\,\,&\int\limits_{B_{\frac{3}{4} r}(x_0)}u^-_{k_\epsilon}(x, \tau)\,\bigg(\int\limits_{\mathbb{R}^N}\frac{\big[u^+_{k_\epsilon}(y, \tau)\big]^{p-1}}{|x-y|^{N+sp}}dy\bigg)\,dxd\tau\leqslant \gamma \int\limits_{B_{r}(x_0)\times\{t_0-\delta \eta\}}\,\,\int\limits_u^{k_\epsilon}\,z^{q-1}(k_\epsilon-z) dz  dx\\&+
\gamma\int\limits^{t_0}_{t_0-\delta \eta}\,\,\int\limits_{B_r(x_0)}\int\limits_{B_r(x_0)}\max[u^-_{k_\epsilon}(x, \tau), u^-_{k_\epsilon}(y, \tau)]^p\frac{|\zeta(x)-\zeta(y)|^p}{|x-y|^{N+sp}}\,dxdyd\tau\\&+
\int\limits^{t_0}_{t_0-\delta \eta}\,\,\int\limits_{B_{\frac{7}{8}r}(x_0)}u^-_{k_\epsilon}(x, \tau)\bigg(\sup\limits_{x\in B_{\frac{7}{8} r}(x_0)}\int\limits_{\mathbb{R}^N\setminus B_r(x_0)}\frac{\big[ u^-_{k_\epsilon}(y, \tau)\big]^{p-1}}{|x-y|^{N+sp}} dy\bigg) dxd\tau\leqslant
 \gamma (\epsilon k)^{q+1}|B_r(x_0)|\\&+ \frac{\gamma (\epsilon k)^p}{r^{sp}} |Q^-_{r, \eta}(x_0, t_0)|\leqslant \gamma (\epsilon k)^{q+1}|B_r(x_0)|,
 \end{aligned}
\end{equation*}
To obtain a lower bound for the left-hand side of the inequality above, observe that on the set $B_{\frac34r}(x_0)
\cap\{u(\cdot,\tau)\leqslant\frac12\epsilon k\},$ we have $u^-_{k\epsilon}(x,\tau)\geqslant\frac{1}{2}\epsilon k.$ Furthermore, for $x\in B_{\frac34r}(x_0)$ and $y\in\mathbb R^N$,
$$
|x-y|
\leqslant
|x-x_0|+|x_0-y|
\leqslant
r+|x_0-y|.
$$
Therefore,
\begin{multline*}
\int\limits^{t_0}_{t_0-\delta \eta}\,\,\int\limits_{B_{\frac{3}{4} r}(x_0)}u^-_{k_\epsilon}(x, \tau)\,\bigg(\int\limits_{\mathbb{R}^N}\frac{\big[u^+_{k_\epsilon}(y, \tau)\big]^{p-1}}{|x-y|^{N+sp}}dy\bigg)\,dxd\tau\\\geqslant
\gamma \epsilon k \,\inf\limits_{t\in (t_0-\delta \eta, t_0)} \big|B_{\frac{3}{4} r}(x_0)\cap\big\{u(\cdot, t) \leqslant \frac{1}{2}\epsilon k\big\}\big|\int\limits^{t_0}_{t_0-\delta \eta}\,\,\int\limits_{\mathbb{R}^N}\frac{\big[u^+_{k_\epsilon}(y, \tau)\big]^{p-1}}{\big(r+|x_0-y|\big)^{N+sp}}dy d\tau.
\end{multline*}
A direct calculation gives
\begin{equation*}
\int\limits^{t_0}_{t_0-\delta \eta}\int\limits_{\mathbb{R}^N}\frac{\big[u^+_{k_\epsilon}(y, \tau)\big]^{p-1}}{\big(r+|x_0-y|\big)^{N+sp}}dy d\tau\geqslant  \gamma \int\limits^{t_0}_{t_0-\delta \eta}\int\limits_{\mathbb{R}^N}\frac{\big[u_+(y, \tau)\big]^{p-1}}{\big(r+|x_0-y|\big)^{N+sp}}dy d\tau -\gamma (\epsilon k)^q.
\end{equation*}
Combining the preceding estimates, we arrive at
\begin{equation}\label{eqShr8}
\inf\limits_{t\in (t_0-\delta \eta, t_0)} \big|B_{\frac{3}{4} r}(x_0)\cap\big\{u(\cdot, t)\leqslant\frac{1}{2}\epsilon k\big\}\big| \leqslant\frac{\gamma (\epsilon k)^q}{\int\limits^{t_0}_{t_0-\delta \eta}\int\limits_{\mathbb{R}^N}\frac{[u_+(y, \tau)]^{p-1}}{(r+|x_0-y|)^{N+sp}}dy d\tau}|B_r(x_0)|,
\end{equation}
which completes the proof of the lemma.
\end{proof}

\subsection{Expansion of Positivity}
The following proposition will be used in the sequel (for the case $q=1$ and $p >2$ we refer the reader to \cite{Lia2026arxiv}).

\begin{proposition}\label{prExpDoblyDeg}
Let $u$ be a  local weak super-solution to \eqref{eq1.1}-\eqref{eq1.3}, $u \geqslant 0$ in $\Omega_T$ and let $0<q< p-1$. 
Assume that
\begin{equation}\label{eqExpDeg1}
|B_{\frac{3}{4} r}(x_0)\cap\big\{u(\cdot, t_0)\geqslant k\big\}|\geqslant \alpha|B_{\frac{3}{4} r}(x_0)|,
\end{equation}
 with some  $k> 0$ and $\alpha \in (0, 1)$. Then there exist numbers  $\delta_1\in (0, 1)$ depending only on the data and $\epsilon_*$, $\xi_*\in (0, 1)$, depending additionally on  $\alpha$,  such that,  if
\begin{equation}\label{eqExpDeg2}
4\,\operatorname{Tail}(u_{-}, \mathcal{Q}) \leqslant\,\epsilon_*\,k,
\end{equation}
then
\begin{equation}\label{eqExpDeg3}
u(x, t)\geqslant \frac{1}{4}\xi_*\,k,\quad  (x, t)\in B_{\frac{r}{2}}(x_0)\times (t_0+\frac{3}{4}\delta_1 (\xi_*\,k)^{q-p+1} r^{sp}, t_0+ \delta_1 (\xi_*\,k)^{q-p+1} r^{sp}),
\end{equation}
provided that $B_{4r}(x_0)\times (t_0, t_0+\delta_1 (\xi_*\,k)^{q-p+1} r^{sp})\subset \mathcal{Q}$.
\end{proposition}

\begin{proof}
The lemma on propagation of positivity in measure, Lemma \ref{lemPropagation}, ensures the existence of $\xi$, $\delta \in (0, 1)$ depending only on the data, such that
\begin{equation}\label{eqExpDeg4}
|B_{\frac{3}{4} r}(x_0)\cap\big\{u(\cdot, t)\geqslant \xi\, k\big\}|\geqslant \frac{1}{4}\alpha|B_{\frac{3}{4} r}(x_0)|,
\quad \text{for all}\quad t\in (t_0, t_0+\eta),\quad \eta:=\delta k^{q-p+1} r^{sp},
\end{equation}
provided that $\epsilon_*\leqslant \delta$.
The measure shrinking lemma, Lemma \ref{lemShrinking} implies that
$$\big|B_{\frac{3}{4} r}(x_0)\cap\big\{u(\cdot, \tilde{t})\leqslant \xi_*\, k\big\}\big|\leqslant \frac{\gamma \xi^q_*}{\xi^{p-1-q}\delta \alpha}|B_{\frac{3}{4} r}(x_0)|,$$
for some time level $\tilde{t}\in (t_0, t_0+\delta k^{q-p+1} r^{sp})$.  Now we use De Giorgi-type lemma involving the measure data, Lemma \ref{lemDGmeasuredata}, which yields 
$$u\geqslant \frac{1}{2}\xi_* k,\quad \text{in}\quad B_{\frac{1}{2} r}(x_0)\times (\tilde{t}+\frac{1}{2}\delta_1 (\xi_* k)^{q-p+1} r^{sp}, \tilde{t}+\delta_1 (\xi_* k)^{q-p+1} r^{sp}),$$
with some $\xi_*$, $\delta_1 \in (0, 1)$ depending only on the data,
provided that 
$$\frac{\gamma \xi^q_*}{\xi^{p-1-q}\delta \alpha} \leqslant \nu_1,\quad \text{and}\quad \epsilon_*\leqslant \frac{1}{\gamma}\xi_*,$$ 
with  $\nu_1 \in (0, 1)$ depending only on the data. Choosing additionally $\xi_*$ from the condition
$$\xi_*\leqslant \Big(\frac{\delta_1}{4 \delta}\Big)^{\frac{1}{p-1-q}},$$
we get
$$(\tilde{t}+\frac{1}{2}\delta_1 (\xi_* k)^{q-p+1} r^{sp}, \tilde{t}+\delta_1 (\xi_* k)^{q-p+1} r^{sp})\supset
(t_0+\frac{3}{4}\delta_1 (\xi_* k)^{q-p+1} r^{sp}, t_0+\delta_1 (\xi_* k)^{q-p+1} r^{sp}).$$
This completes the proof of the proposition.
\end{proof}

\section{Proof of Proposition \ref{pr1.1}}\label{sec:prop1.1}
We recall the definition of $\mathcal{I}$
\begin{equation}\label{eqPR1}
\mathcal{I}\leqslant \bigg(\int\limits^{\bar{t}}_{\bar{t}-\theta}\int\limits_{\mathbb{R}^N}\frac{\big[u_+(x, t)\big]^{p-1}}{\big(\rho+|x-\bar{x}|\big)^{N+sp}}\,dx dt\bigg)^{\frac{1}{q}},\quad \theta=\mathcal{I}^{q-p+1} \rho^{sp}.
\end{equation}
To prove Proposition \ref{pr1.1} we use the measures shrinking lemma, Lemma \ref{lemShrinking1} with $k$ replaced by $\mathcal{I}$, which by \eqref{eqPR1} yields, for every $\epsilon \in (0, 1/2)$ that
\begin{equation*}
\big|B_{\frac{3}{4} \rho}(\bar{x})\cap\big\{u(\cdot, \tilde{t})\leqslant \frac{1}{2} \epsilon \mathcal{I}\big\}\big|\leqslant \gamma\,\epsilon^{q}|B_{\frac{3}{4} \rho}(\bar{x})|,
\end{equation*}
for some time level $\tilde{t}\in (\bar{t}-\theta, \bar{t})$, provided that
$$\operatorname{Tail}(u_-, Q_{\rho, \bar{C}_*\theta}(\bar{x}, \bar{t}))\leqslant \frac{1}{\gamma}\epsilon\,\mathcal{I},$$
with constant $C_*> 1$ to be defined. Choose $\epsilon $ according to the condition $\gamma \epsilon^q \leqslant \nu_1$, where $\nu_1 \in (0, 1)$ is the number whose existence is asserted in De Giorgi-type lemma with initial measure condition, Lemma \ref{lemDGmeasuredata}.
By this lemma, we obtain
$$u\geqslant \frac{1}{4}\epsilon \mathcal{I},\quad \text{in}\quad B_{\frac{1}{2}\rho}(\bar{x})\times (\tilde{t}+\frac{1}{2}\delta_1(\epsilon \mathcal{I})^{q-p+1}\rho^{sp}, \tilde{t}+\delta_1 (\epsilon \mathcal{I})^{q-p+1}\rho^{sp}),$$
with some $\delta_1\in (0, 1)$ depending only on the data. Choosing $\epsilon$ smaller if necessary, such that
$$(\tilde{t}+\frac{1}{2}\delta_1(\epsilon \mathcal{I})^{q-p+1}\rho^{sp}, \tilde{t}+\delta_1 (\epsilon \mathcal{I})^{q-p+1}\rho^{sp})\supset(\bar{t}+
\frac{3}{4}\delta_1(\epsilon \mathcal{I})^{q-p+1}\rho^{sp}, \bar{t}+\delta_1 (\epsilon \mathcal{I})^{q-p+1}\rho^{sp}),$$
and fixing $C_*$ by the condition $C_*\geqslant \delta_1\,\epsilon^{q-p+1}$, we arrive at the required \eqref{eqWeakProp3}. This completes the proof of Proposition \ref{pr1.1}.

\section{Proof of Proposition \ref{pr1.2}}\label{sec:prop1.2}

In this section, we prove the upper bound for local sub-solutions to \eqref{eq1.1}-\eqref{eq1.3}, with $u \geqslant 0$ in $\Omega_T$. Throughout this section, we assume that all cylinders introduced below are contained in $\Omega_T$.

For $k >0$, $\sigma \in (0, 1)$, $0< r< R$, and $0< \lambda < \Lambda$ set
$$
\begin{aligned}
r_\sigma&:=(1-\sigma)r+\sigma R,
&\qquad
\bar r_\sigma&:=\frac{r_\sigma+r_{\sigma/2}}{2},
\\
\eta_\sigma&:=(1-\sigma)\lambda+\sigma\Lambda,
&\qquad
\bar\eta_\sigma&:=\frac{\eta_\sigma+\eta_{\sigma/2}}{2},
\end{aligned}
$$
and define the time-dependent levels
$$\begin{aligned}
k_\sigma(t):= k(1-\sigma) -\tilde{C}\,\mathcal{K}(t),\quad \bar{k}_\sigma(t):=\frac{k_\sigma(t)+k_{\sigma/2}(t)}{2},
\end{aligned}$$
where the constant $\tilde{C} >0$ will be chosen later and
$$\mathcal{K}(t):= k^{1-q}
\int\limits^{t}_{\bar{t}-\Lambda}\,\,\,\int\limits_{\mathbb{R}^N\setminus B_R(\bar{x})}\frac{[u_+(x, \tau)]^{p-1}}{|x-\bar{x}|^{N+sp}}dxd\tau.$$
We also set
$$
\begin{aligned}
B_\sigma&:=B_{r_\sigma}(\bar{x}),
&\qquad
\bar{B}_\sigma&:=B_{\bar{r}_\sigma}(\bar{x}),
\\
Q_\sigma&:=B_\sigma\times(\bar{t}-\eta_\sigma, \bar{t}),
&\qquad
\bar{Q}_\sigma&:=\bar{B}_\sigma\times (\bar{t}-\bar{\eta}_\sigma, \bar{t}k),
\end{aligned}
$$
Define the cutoff function $\zeta(x, t):=\zeta_1(x)\zeta_2(t)$, where $\zeta_1(x)\in C^1_0(\bar{B}_\sigma)$ and $\zeta_2(t)\in C^1(\mathbb{R})$ satisfy the following conditions:
$$
0\leqslant \zeta_1(x)\leqslant 1, \quad \zeta_1(x)=1 \quad \text{in}\quad B_{\sigma/2},\quad \text{and}\quad |\nabla \zeta_1(x)|\leqslant\frac{2}{\sigma (R-r)},$$
$$
 0\leqslant \zeta_2(t)\leqslant 1,\quad \zeta_2(t)=
\begin{cases}
0, & t\leqslant  \bar{t}-\bar{\eta}_{\sigma},\\
1, & t\geqslant \bar{t}-\eta_{\sigma/2},
\end{cases} \quad
|\zeta'_2(t)|\leqslant \frac{2}{\sigma (\Lambda-\lambda)}.
$$
We further assume that
\begin{equation}\label{eqUpper1}
k\geqslant 4 \tilde{C}\,\mathcal{K}(\bar{t}),\quad \text{i.e.,}\quad  k\geqslant \bigg(4 \tilde{C}\int\limits^{\bar{t}}_{\bar{t}-\Lambda}\,\,\,\int\limits_{\mathbb{R}^N\setminus B_R(\bar{x})}\frac{[u_+(x, \tau)]^{p-1}}{|x-\bar{x}|^{N+sp}}dxd\tau\bigg)^{\frac{1}{q}}.
\end{equation} 
We apply the energy estimate \eqref{eqenergy} to the function  $u^+_{\bar{k}_\sigma(t)}=(u-\bar{k}_\sigma(t))_+$  and obtain
\begin{equation}\label{eqUpper2}
\begin{aligned}
 &\sup\limits_{\bar{t}-\bar{\eta}_{\sigma}\leqslant \tau \leqslant \bar{t}}\int\limits_{\bar{B}_\sigma\times\{\tau\}}\,\,\int\limits^u_{\bar{k}_\sigma(\tau)}
 z^{q-1}(z-\bar{k}_\sigma(\tau))\,dz\zeta^p\,dx\\&+\int\limits^{\bar{t}}_{\bar{t}-\bar{\eta}_\sigma}\,\,\int\limits_{\bar{B}_\sigma}
 \int\limits_{\bar{B}_\sigma}
 \frac{|u^+_{\bar{k}_\sigma(\tau)}(x, \tau) -u^+_{\bar{k}_\sigma(\tau)}(y, \tau)|^p}{|x-y|^{N+sp}}\min[\zeta(x, \tau), \zeta(y, \tau)]^p\,dxdy\,d\tau
 \\&\leqslant -\tilde{C}^q\,\iint\limits_{Q_\sigma}\mathcal{K}'(\tau)\,\int\limits^u_{\bar{k}_\sigma(\tau)}z^{q-1}\,
 dz\zeta^p\,dxd\tau+ \gamma \iint\limits_{Q_\sigma}\int\limits^u_{\bar{k}_\sigma(\tau)}
 z^{q-1}(z-\bar{k}_\sigma(\tau))\,dz\zeta^{p-1}\,|\zeta'_\tau|dxd\tau\\&+\gamma \iint\limits_{Q_\sigma}\int\limits_{B_\sigma}\big[u^+_{\bar{k}_\sigma(\tau)}(x, \tau)+u^+_{\bar{k}_\sigma(\tau)}(y, \tau)\big]^p\,\frac{|\zeta(x, \tau)-\zeta(y, \tau)|^p}{|x-y|^{N+sp}}dx dy d\tau\\&+\gamma \iint\limits_{Q_\sigma}
 u^+_{\bar{k}_\sigma(\tau)}(x, \tau)\zeta^p(x, \tau)\bigg(\int\limits_{\mathbb{R}^N\setminus B_\sigma}\frac{\big[u^+_{\bar{k}_\sigma(\tau)}(y, \tau)\big]^{p-1}}{|x-y|^{N+sp}} dy\bigg)dxd\tau=\mathcal{J}_1+\mathcal{J}_2+\mathcal{J}_3+\mathcal{J}_4.
 \end{aligned}
\end{equation} 
On the set $Q_\sigma \cap\big\{u\geqslant \bar{k}_{\sigma}(\cdot)\big\}$, we have
$$u\leqslant \frac{\gamma}{\sigma}(u-k_\sigma(t))_+.$$
Therefore, $\mathcal J_2$ and $\mathcal J_3$ can be estimated as follows:
\begin{equation*}
\mathcal{J}_2 +\mathcal{J}_3 \leqslant \frac{\gamma}{\sigma^{q+1}(\Lambda-\lambda)}\iint\limits_{Q_\sigma}[u^+_{k_\sigma(t)}]^{q+1}\,dxdt+
\frac{\gamma}{\sigma^{p}(R-r)^{sp}}\iint\limits_{Q_\sigma}[u^+_{k_\sigma(t)}]^p\,dx dt.
\end{equation*}
Choosing $k$ such that
\begin{equation}\label{eqUpper3}
k\geqslant \bigg(\frac{(R-r)^{sp}}{\Lambda-\lambda}\bigg)^{\frac{1}{p-1-q}},
\end{equation}
we obtain
\begin{equation*}
\mathcal{J}_2 +\mathcal{J}_3 \leqslant \frac{\gamma}{\sigma^{p}(R-r)^{sp}}\iint\limits_{Q_\sigma}[u^+_{k_\sigma(t)}]^p\,dx dt.
\end{equation*}
Let $\bar{\epsilon}\in(0,1)$ be specified later.  In the case $0< q< 1$, we additionally assume that
\begin{equation}\label{eqUpper4}
k\geqslant \bar{\epsilon} M_1,\quad M_\sigma:=\sup\limits_{Q_\sigma}\,u.
\end{equation}
By the algebraic Lemma~\ref{lemAlg}, we obtain
\begin{equation}\label{eqUpper5}
\begin{cases}
\int\limits^u_{\bar{k}_\sigma(\tau)}z^{q-1}\,dz \geqslant \gamma\,u^+_{\bar{k}_\sigma(\tau)}\Big[u+\bar{k}_\sigma(\tau)\Big]^{q-1}\geqslant \gamma u^+_{\bar{k}_\sigma(\tau)}\,k^{q-1},\quad \text{if}\quad q\geqslant 1,\\
\int\limits^u_{\bar{k}_\sigma(\tau)}z^{q-1}\,dz \geqslant \gamma u^+_{\bar{k}_\sigma(\tau)}\Big[u+\bar{k}_\sigma(\tau)\Big]^{q-1}\geqslant \gamma u^+_{\bar{k}_\sigma(\tau)}\,\Big[M_\sigma+k\Big]^{q-1}\geqslant \gamma \bar{\epsilon}^{1-q}u^+_{\bar{k}_\sigma(\tau)}\,k^{q-1},\,\, \text{if}\,\, 0< q< 1.
\end{cases}
\end{equation}
Therefore, in both cases we have
\begin{equation*}
\begin{aligned}
\mathcal{J}_1&\leqslant -\tilde{C}\bar{\epsilon} \frac{k^{q-1}}{\gamma}\iint\limits_{Q_\sigma}u^+_{\bar{k}_\sigma(\tau)}\,\mathcal{K}'(\tau)\zeta^p\,dxd\tau
\\&=-\frac{\tilde{C}\bar{\epsilon}}{\gamma}\iint\limits_{Q_\sigma}u^+_{\bar{k}_\sigma(\tau)}(x, \tau)\,\zeta^p(x, \tau)\bigg(\int\limits_{\mathbb{R}^N\setminus B_R(\bar{x})}\frac{[u_+(y, \tau)]^{p-1}}{|y-\bar{x}|^{N+sp}}\,dy\bigg)\,dxd\tau.
\end{aligned}
\end{equation*}
Finally, using the facts that
$$|y-x|\geqslant \frac{R-\bar{r}_\sigma}{R} |y-\bar{x}|\geqslant \frac{(1-\frac{3}{4}\sigma)(R-r)}{R}|y-\bar{x}|\geqslant \frac{1}{4}\frac{R-r}{R}|y-\bar{x}|,\,\, x\in \bar{B}_\sigma,\,\, y\in \mathbb{R}^N\setminus B_R(\bar{x}),$$
and
$$|y-x|\geqslant r_\sigma-\bar{r}_\sigma= \frac{1}{4}\sigma\,(R-r),\quad x\in \bar{B}_\sigma,\quad y\in \mathbb{R}^N\setminus B_\sigma,$$
we obtain
\begin{equation*}
\begin{aligned}
\mathcal{J}_4&\leqslant\gamma \iint\limits_{Q_\sigma}
 u^+_{\bar{k}_\sigma(\tau)}(x, \tau)\zeta^p(x, \tau)\bigg(\sup\limits_{x\in \bar{B}_\sigma}\int\limits_{\mathbb{R}^N\setminus B_R(\bar{x})}\frac{\big[u^+_{\bar{k}_\sigma(\tau)}(y, \tau)\big]^{p-1}}{|x-y|^{N+sp}} dy\bigg)dxd\tau\\&+\gamma \iint\limits_{Q_\sigma}
 u^+_{\bar{k}_\sigma(\tau)}(x, \tau)\zeta^p(x, \tau)\bigg(\sup\limits_{x\in \bar{B}_\sigma}\int\limits_{B_R(\bar{x})\setminus B_\sigma}\frac{\big[u^+_{\bar{k}_\sigma(\tau)}(y, \tau)\big]^{p-1}}{|x-y|^{N+sp}} dy\bigg)dxd\tau\\
&\leqslant
 \gamma \Big(\frac{R}{R-r}\Big)^{N+sp}\iint\limits_{Q_\sigma}
 u^+_{\bar{k}_\sigma(\tau)}(x, \tau)\zeta^p(x, \tau)\bigg(\frac{k^{p-1}}{R^{sp}}+\int\limits_{\mathbb{R}^N\setminus B_R(\bar{x})}\frac{\big[u_+(y, \tau)\big]^{p-1}}{|y-\bar{x}|^{N+sp}} dy\bigg)dxd\tau\\&+\frac{\gamma M^{p-1}_1 R^N}{(\sigma(R-r))^{N+sp}}\iint\limits_{Q_\sigma}
 u^+_{\bar{k}_\sigma(\tau)}(x, \tau)dxd\tau\\&\leqslant \gamma \Big(\frac{R}{R-r}\Big)^{N+sp}\iint\limits_{Q_\sigma}
 u^+_{\bar{k}_\sigma(\tau)}(x, \tau)\zeta^p(x, \tau)\bigg(\int\limits_{\mathbb{R}^N\setminus B_R(\bar{x})}\frac{\big[u_+(y, \tau)\big]^{p-1}}{|y-\bar{x}|^{N+sp}} dy\bigg)dxd\tau\\&+ \gamma \frac{\bar{\epsilon}^{1-p}}{\sigma^{N+sp}} \Big(\frac{R}{R-r}\Big)^{N+sp} (R-r)^{-sp} k^{p-1}\iint\limits_{Q_\sigma}
 u^+_{\bar{k}_\sigma(\tau)}(x, \tau)dxd\tau\\&\leqslant \gamma \Big(\frac{R}{R-r}\Big)^{N+sp}\iint\limits_{Q_\sigma}
 u^+_{\bar{k}_\sigma(\tau)}(x, \tau)\zeta^p(x, \tau)\bigg(\int\limits_{\mathbb{R}^N\setminus B_R(\bar{x})}\frac{\big[u_+(y, \tau)\big]^{p-1}}{|y-\bar{x}|^{N+sp}} dy\bigg)dxd\tau\\&+ \gamma \frac{\bar{\epsilon}^{1-p}(R-r)^{-sp}}{\sigma^{N+(s+1)p-1}} \Big(\frac{R}{R-r}\Big)^{N+sp} \iint\limits_{Q_\sigma}
 [u^+_{k_\sigma(\tau)}]^pdxd\tau.
 \end{aligned}
\end{equation*}
The first term on the right-hand side of this inequality is cancelled by the right-hand side of the estimate for $\mathcal{J}_1$, provided that $\tilde{C}$ is chosen to satisfy
 $\frac{\tilde{C} \bar{\epsilon}}{\gamma} = \gamma \big(\frac{R}{R-r}\big)^{N+sp}$.

Next, we estimate the left-hand side of \eqref{eqUpper2}. Using
\eqref{eqUpper5}
\begin{equation*}
\begin{aligned}
\int\limits^{u}_{\bar{k}_\sigma(\tau)}z^{q-1}(z-\bar{k}_\sigma(\tau)) dz&\geqslant \int\limits^{u}_{\frac{1}{2}(\bar{k}_\sigma(\tau)+u)}z^{q-1}(z-\bar{k}_\sigma(\tau))\,dz
\\&\geqslant \frac{1}{2}(u-\bar{k}_\sigma(\tau))\,\int\limits^{u}_{\frac{1}{2}(\bar{k}_\sigma(\tau)+u)}z^{q-1}\,dz
\geqslant \frac{\bar{\epsilon}}{\gamma}k^{q-1}\,\big[u^+_{\bar{k}_\sigma(\tau)}\big]^2.
\end{aligned}
\end{equation*}
Therefore, collecting the previous estimates, inequality \eqref{eqUpper2} can be rewritten as
\begin{equation}\label{eqUpper6}
\begin{aligned}
&k^{q-1}\,\sup\limits_{\bar{t}-\bar{\eta}_{\sigma}\leqslant \tau \leqslant \bar{t}}\,\,\,\int\limits_{\bar{B}_\sigma\times\{\tau\}}\big[u^+_{\bar{k}_\sigma(\tau)}\big]^2\,\zeta^p\,dx\\&+
\int\limits^{\bar{t}}_{\bar{t}-\bar{\eta}_\sigma}\,\,\int\limits_{\bar{B}_\sigma}\int\limits_{\bar{B}_\sigma}
 \frac{|u^+_{\bar{k}_\sigma(\tau)}(x, \tau) -u^+_{\bar{k}_\sigma(\tau)}(y, \tau)|^p}{|x-y|^{N+sp}}\min[\zeta(x, \tau), \zeta(y, \tau)]^p\,dxdy\,d\tau
\\&\leqslant \frac{\gamma}{\sigma^\gamma \bar{\epsilon}^\gamma}\,(R-r)^{-sp} \Big(\frac{R}{R-r}\Big)^{N+sp} \iint\limits_{Q_\sigma}
 [u^+_{k_\sigma(\tau)}]^p dxd\tau.
 \end{aligned}
\end{equation}
Now, using the Sobolev embedding theorem, see Lemma \ref{lemEmbedding}, we obtain 
\begin{multline*}
\qquad\qquad\qquad\iint\limits_{Q_{\sigma/2}}[u^+_{k_{\sigma/2}(\tau)}]^p dxd\tau\leqslant
\Big(\iint\limits_{\bar{Q}_\sigma}\big[u^+_{\bar{k}_\sigma(\tau)}(x, \tau)\big]^{p(1+\frac{2s}{N})}[\zeta(x, \tau)]^{p(1+\frac{sp}{N})}\,dxd\tau\Big)^{\frac{N}{N+2s}}\\\times \big|Q_\sigma \cap\big\{u\geqslant k_{\sigma/2}(\cdot)\big\}\big|^{\frac{2s}{N+2s}}\leqslant  \gamma \bigg(\sup\limits_{\bar{t}-\bar{\eta}_{\sigma}\leqslant \tau \leqslant \bar{t}}\int\limits_{\bar{B}_\sigma\times\{\tau\}}\big[u^+_{\bar{k}_\sigma(\tau)}\big]^2\,\zeta^p\,dx\bigg)^{\frac{sp}{N+2s}}\\\times
\bigg(\int\limits^{\bar{t}}_{\bar{t}-\bar{\eta}_\sigma}\int\limits_{\bar{B}_\sigma}\int\limits_{\bar{B}_\sigma}
 \frac{|u^+_{\bar{k}_\sigma(\tau)}(x, \tau)\zeta(x, \tau) -u^+_{\bar{k}_\sigma(\tau)}(y, \tau)\zeta(y,\tau)|^p}{|x-y|^{N+sp}}\,dxdyd\tau
+\frac{\gamma}{\sigma^\gamma(R-r)^{sp}}\iint\limits_{\bar{Q}_\sigma}[u^+_{\bar{k}_\sigma(\tau)}]^p dxd\tau \bigg)^{\frac{N}{N+2s}}\\\times \big|Q_\sigma \cap\big\{u\geqslant k_{\sigma/2}(\cdot)\big\}\big|^{\frac{2s}{N+2s}}.
\end{multline*}
To estimate the nonlocal term in the last inequality, we use
\begin{multline*}
|u^+_{\bar{k}_\sigma(\tau)}(x, \tau)\zeta(x, \tau) -u^+_{\bar{k}_\sigma(\tau)}(y, \tau)\zeta(y,\tau)|^p\leqslant\\\leqslant
\gamma|u^+_{\bar{k}_\sigma(\tau)}(x, \tau) -u^+_{\bar{k}_\sigma(\tau)}(y, \tau)|^p\zeta^p(x, \tau)+\gamma |\zeta(x, \tau) -\zeta(y,\tau)|^p[u^+_{\bar{k}_\sigma(\tau)}(x, \tau)]^p.
\end{multline*}
Substituting this estimate into the preceding inequality, applying \eqref{eqUpper6}, and setting
\[
\mathbf{Y}(\sigma):=
\iint\limits_{Q_{\sigma}}[u^+_{k_{\sigma}(\tau)}]^pdxd\tau,
\]
we arrive at
\begin{equation*}
\mathbf{Y}(\sigma/2)\leqslant \frac{\gamma}{\bar{\epsilon}^\gamma\sigma^\gamma}
k^{-\frac{sp(q+1)}{N+2s}} (R-r)^{-\frac{sp(N+sp)}{N+2s}} \Big(\frac{R}{R-r}\Big)^{\frac{(N+sp)^2}{N+2s}}\big[\mathbf{Y}(\sigma)\big]^{1+\frac{sp}{N+2s}}.
\end{equation*}
Repeating the iteration from Lemma~\ref{lemDG}, with
$\sigma_j:=2^{1-j},\,\, Y_j:=\mathbf{Y}(\sigma_j),  j=1,2,\ldots$, we obtain $Y_j\to0$ as $j\to\infty$, provided that
$$Y_1=\mathbf{Y}(1)\leqslant 
\frac{\bar{\epsilon}^\gamma}{\gamma}\,k^{q+1}(R-r)^{N+sp}\Big(\frac{R-r}{R}\Big)^{\frac{(N+sp)^2}{sp}},
$$
which, together with our choices \eqref{eqUpper1}, \eqref{eqUpper3} and \eqref{eqUpper4}, yields
\begin{equation*}
\begin{aligned}
M_{0}\leqslant& \,\bar{\epsilon}\,M_1+\frac{\gamma}{ \bar{\epsilon}^\gamma}\bigg((R-r)^{-N-sp}\Big(\frac{R}{R-r}\Big)^{\frac{(N+sp)^2}{sp}}\iint\limits_{Q_1} u^p\,dxd\tau\bigg)^{\frac{1}{q+1}}+
\gamma \Big(\frac{(R-r)^{sp}}{\Lambda-\lambda}\Big)^{\frac{1}{p-1-q}}\\&+\frac{\gamma}{\bar{\epsilon}^\gamma} \bigg( \Big(\frac{R}{R-r}\Big)^{N+sp}\int\limits^{\bar{t}}_{\bar{t}-\Lambda}\,\,\,\int\limits_{\mathbb{R}^N\setminus B_R(\bar{x})}\frac{[u_+(x, \tau)]^{p-1}}{|x-\bar{x}|^{N+sp}}dxd\tau\bigg)^{\frac{1}{q}}.
\end{aligned}
\end{equation*}
Using Young's inequality, we arrive at
\begin{equation}\label{eqUpper7}
\begin{aligned}
M_{0}\leqslant&\, \bar{\epsilon}\,M_1+\frac{\gamma}{ \bar{\epsilon}^\gamma}\bigg((R-r)^{-N-sp}\Big(\frac{R}{R-r}\Big)^{\frac{(N+sp)^2}{sp}}\iint\limits_{Q_1} u^{p-1}\,dxd\tau\bigg)^{\frac{1}{q}}+
\gamma \Big(\frac{(R-r)^{sp}}{\Lambda-\lambda}\Big)^{\frac{1}{p-1-q}}\\&+\frac{\gamma}{\bar{\epsilon}^\gamma} \bigg( \Big(\frac{R}{R-r}\Big)^{N+sp}\int\limits^{\bar{t}}_{\bar{t}-\Lambda}\,\,\,\int\limits_{\mathbb{R}^N\setminus B_R(\bar{x})}\frac{[u_+(x, \tau)]^{p-1}}{|x-\bar{x}|^{N+sp}}dxd\tau\bigg)^{\frac{1}{q}}.
\end{aligned}
\end{equation}
Next, we use \eqref{eqUpper7} in an interpolation argument.
For fixed $\rho$, $\theta>0$ and $\sigma \in (0, 1)$, choose $r$ and $\lambda$ such that $\frac{1}{2}\rho \leqslant r < r(1+\sigma) \leqslant \rho$ and $\frac{1}{2} \theta \leqslant \lambda < \lambda (1+\sigma)\leqslant \theta$.
Applying \eqref{eqUpper7} with $R=r(1+\sigma)$ and $\Lambda=\lambda(1+\sigma)$ yields
\begin{equation*}
\begin{aligned}
\sup\limits_{Q^-_{r, \lambda}(\bar{x}, \bar{t})} u&\leqslant \bar{\epsilon}\,\sup\limits_{Q^-_{r(1+\sigma), \lambda(1+\sigma)}(\bar{x}, \bar{t})} u+\frac{\gamma}{ \bar{\epsilon}^\gamma \sigma^\gamma}\bigg[\bigg(\rho^{-N-sp}\iint\limits_{Q^-_{\rho, \theta}(\bar{x}, \bar{t})} u^{p-1}\,dxd\tau\bigg)^{\frac{1}{q}}+
 \Big(\frac{\rho^{sp}}{\theta}\Big)^{\frac{1}{p-1-q}}\\&+ \bigg( \int\limits^{\bar{t}}_{\bar{t}-\theta}\,\,\,\int\limits_{\mathbb{R}^N\setminus B_{\frac{1}{2}\rho}(\bar{x})}\frac{[u_+(x, \tau)]^{p-1}}{|x-\bar{x}|^{N+sp}}dxd\tau\bigg)^{\frac{1}{q}}\bigg].
\end{aligned}
\end{equation*}
By interpolation, we obtain
\begin{equation*}
\begin{aligned}
\sup\limits_{Q^-_{\frac{1}{2}\rho, \frac{1}{2}\theta}(\bar{x}, \bar{t})} u&\leqslant \gamma\bigg[\bigg(\rho^{-N-sp}\iint\limits_{Q^-_{\rho, \theta}(\bar{x}, \bar{t})} u^{p-1}\,dxd\tau\bigg)^{\frac{1}{q}}+ \bigg( \int\limits^{\bar{t}}_{\bar{t}-\theta}\int\limits_{\mathbb{R}^N\setminus B_{\frac{1}{2}\rho}(\bar{x})}\frac{[u_+(x, \tau)]^{p-1}}{|x-\bar{x}|^{N+sp}}dxd\tau\bigg)^{\frac{1}{q}}\\&+
 \Big(\frac{\rho^{sp}}{\theta}\Big)^{\frac{1}{p-1-q}}\bigg]\leqslant \gamma \bigg[\bigg( \int\limits^{\bar{t}}_{\bar{t}-\theta}\,\,\,\int\limits_{\mathbb{R}^N}\frac{[u_+(x, \tau)]^{p-1}}{(\rho+|x-\bar{x}|)^{N+sp}}dxd\tau\bigg)^{\frac{1}{q}}+
 \Big(\frac{\rho^{sp}}{\theta}\Big)^{\frac{1}{p-1-q}}\bigg].
\end{aligned}
\end{equation*}
This completes the proof of Proposition \ref{pr1.2}.

\section{Proof of Theorem \ref{th1.1}}\label{sec:mainProof}

We first prove the right-hand inequality in \eqref{eqHarnack2} of Theorem~\ref{th1.1}.  We apply Proposition~\ref{pr1.2} in the cylinder
$$Q^-_{\rho, \theta}(x_0, t_0)=B_\rho(x_0)\times(t_0-\theta, t_0),\qquad \theta=\Big(\frac{1}{2\bar{C}_3} u(x_0, t_0)\Big)^{q-p+1}\rho^{sp},$$
where $\bar{C}_3 >0$ is the constant from Proposition \ref{pr1.2}. This gives
\begin{equation}\label{eq6.1}
u(x_0, t_0)\leqslant \sup\limits_{Q^-_{\frac{1}{2}\rho, \frac{1}{2}\theta}(x_0, t_0)} u\leqslant 2\bar{C}_3  \bigg( \int\limits^{t_0}_{t_0-\theta}\,\,\int\limits_{\mathbb{R}^N}\frac{[u_+(x, \tau)]^{p-1}}{(\rho+|x-x_0|)^{N+sp}}dxd\tau\bigg)^{\frac{1}{q}}, 
\end{equation}
which yields the right-hand side of  \eqref{eqHarnack2}. To prove the right-hand side of \eqref{eqHarnack1}, we apply Proposition \ref{pr1.1}. In view of \eqref{eq6.1}, we set
$$
\mathcal I:=\frac{1}{2\bar{C}_3}u(x_0,t_0).
$$
Then
$$ \mathcal{I}\leqslant \bigg( \int\limits^{t_0}_{t_0-\theta}\,\,\int\limits_{\mathbb{R}^N}\frac{[u_+(x, \tau)]^{p-1}}{(\rho+|x-x_0|)^{N+sp}}dxd\tau\bigg)^{\frac{1}{q}},\quad \theta=\mathcal{I}^{q-p+1} \rho^{sp}.$$
Applying \eqref{eqWeakProp3} from Proposition \ref{pr1.1}, we arrive at the required  inequality  \eqref{eqHarnack1}, namely,
\begin{equation}\label{eq6.2}
\frac{1}{2\bar{C}_3}u(x_0,t_0)=\mathcal{I}\leqslant \bar{C}_2\,\inf\limits_{B_\rho(x_0)} u(\cdot, t_0+\bar{C}_* \Big(\frac{1}{2\bar{C}_3}u(x_0, t_0)\Big)^{q-p+1}\rho^{sp}),
\end{equation}
provided that
\begin{equation*}
\left(
\int\limits_{t_0-4\rho^{sp}\bar{C}_*(\frac{1}{2\bar{C}_3}u(x_0, t_0))^{q-p+1} }^{t_0+4\rho^{sp}\bar{C}_*(\frac{1}{2\bar{C}_3}u(x_0, t_0))^{q-p+1}}
\,\,\int\limits_{\mathbb{R}^N\setminus B_\rho(x_0)}
\frac{[u_-(x,t)]^{p-1}}
{|x-x_0|^{N+sp}}
,dx,dt
\right)^{\frac1q}
\leqslant
\frac{1}{2\bar{C}_1\bar{C}_3} u(x_0, t_0).
\end{equation*}
To prove the left-hand inequality in \eqref{eqHarnack1}, we use the results of \cite[Chapter 5, Section 3]{DiBGiaVes2012}, which require only the continuity of the solution and the right-hand inequality in \eqref{eqHarnack1}, established above. We leave the details of the proof to the reader.

To establish the left-hand inequality in \eqref{eqHarnack2}, we apply inequality \eqref{eqShr8} from the shrinking Lemma \ref{lemShrinking1}  with $\frac{1}{2} \epsilon k=C_2 u(x_0, t_0)$, in the cylinder $Q^*:=B_\rho(x_0)\times (t_0-C_*\eta, t_0-\frac{1}{2} C_* \eta)$, where
$\eta:=[u(x_0, t_0)]^{q-p+1}\rho^{sp}$  and $C_2$, $C_*> 0$ are the constants defined in \eqref{eqHarnack1} of Theorem \ref{th1.1}. We obtain
\begin{equation*}
\inf\limits_{t\in (t_0-C_*\eta, t_0-\frac{1}{2}C_*\eta)} \big|B_{\rho}(x_0)\cap\big\{u(\cdot, t)\leqslant C_2 u(x_0, t_0)\big\}\big| \leqslant\frac{\gamma(C_2) [u(x_0, t_0)]^q}{\int\limits^{t_0-\frac{1}{2}C_* \eta}_{t_0-C_*\eta}\int\limits_{\mathbb{R}^N}\frac{[u_+(y, \tau)]^{p-1}}{(\rho+|x_0-y|)^{N+sp}}dy d\tau}|B_\rho(x_0)|.
\end{equation*}
By the left-hand inequality in \eqref{eqHarnack1}, we have
$$
\inf\limits_{t\in (t_0-C_*\eta, t_0-\frac{1}{2}C_*\eta)} \big|B_{\rho}(x_0)\cap\big\{u(\cdot, t)\leqslant C_2 u(x_0, t_0)\big\}\big|=|B_\rho(x_0)|.
$$
Combining this with the preceding estimate, we obtain
$$
\int\limits^{t_0-\frac{1}{2}C_* \eta}_{t_0-C_*\eta}\int\limits_{\mathbb{R}^N}\frac{[u_+(y, \tau)]^{p-1}}{(\rho+|x_0-y|)^{N+sp}}dy d\tau\leqslant \gamma [u(x_0, t_0)]^q.$$
This completes the proof of Theorem \ref{th1.1}.

\section*{Acknowledgements} 
Igor Skrypnik and Mariia Savchenko were supported by the National
Research Foundation of Ukraine, project No. 2025.07/0369 “Qualitative methods of nonlinear
analysis of heterogeneous structures.”


\bibliographystyle{abbrvurl}
\bibliography{ref}

\end{document}